\documentclass[12pt,letterpaper]{article}

\pdfoutput=1

\usepackage{graphics,epsfig,graphicx,color}
\graphicspath{{Figures/}}
\usepackage{float}
\usepackage{capt-of}
\usepackage{subfig}

\usepackage{amssymb,latexsym}

\usepackage{setspace}

\usepackage{amsmath}

\usepackage{mathrsfs,amsfonts}

\usepackage{amsthm,amsxtra}

\usepackage[final]{showkeys}

\usepackage{stmaryrd}

\usepackage{algorithm}
\usepackage{algorithmicx}
\usepackage{algpseudocode}

\usepackage[openbib]{currvita}

\usepackage{etaremune}
\usepackage{enumitem} 

\newif\ifPDF
\ifx\pdfoutput\undefined
	\PDFfalse
\else
	\ifnum\pdfoutput > 0
		\PDFtrue
	\else
		\PDFfalse
	\fi
\fi

\ifPDF
	\usepackage{pdftricks}
	\begin{psinputs}
		\usepackage{pstricks}
		\usepackage{pstcol}
		\usepackage{pst-plot}
		\usepackage{pst-tree}
		\usepackage{pst-eps}
		\usepackage{multido}
		\usepackage{pst-node}
		\usepackage{pst-eps}
	\end{psinputs}
\else
	\usepackage{pstricks}
\fi

\ifPDF
	\usepackage[debug,pdftex,colorlinks=true, 
	linkcolor=blue, bookmarksopen=false,
	plainpages=false,pdfpagelabels]{hyperref}
\else
	\usepackage[dvips]{hyperref}
\fi

\usepackage{fancyhdr}

\usepackage{comment}

\usepackage[toc,page]{appendix}

\usepackage{cancel}

\usepackage{multirow}

\newtheorem{theorem}{Theorem}[section]

\newtheorem{remark}[theorem]{Remark}

\newcommand{\supp}{\operatorname{supp}}

\newcommand{\dint}{\displaystyle\int}

\newcommand{\eps}{\varepsilon}

\newcommand{\fc}{\mathfrak{c}}

\newcommand{\bbR}{\mathbb R}

\newcommand{\bxi}{\boldsymbol \xi}

 \newcommand{\bn}{\mathbf n}

\newcommand{\bu}{\mathbf u} \newcommand{\bv}{\mathbf v} 
\newcommand{\bw}{\mathbf w} \newcommand{\bx}{\mathbf x} 
\newcommand{\by}{\mathbf y}

\newcommand{\bI}{\mathbf I} 
 
\newcommand{\bM}{\mathbf M}

 \newcommand{\bT}{\mathbf T}

\newcommand{\cC}{\mathcal C}  
 \newcommand{\cF}{\mathcal F}
\newcommand{\cG}{\mathcal G} \newcommand{\cH}{\mathcal H}
\newcommand{\cI}{\mathcal I} 
 
\newcommand{\cM}{\mathcal M} 
\newcommand{\cO}{\mathcal O}  
 
 \newcommand{\cT}{\mathcal T}

\newcommand{\aver}[1]{\langle {#1} \rangle}

\DeclareMathOperator*{\argmin}{arg\,min}

\newcommand{\wh}{\widehat}

\newenvironment{keywords}
{\noindent{\bf Key words.}\small}{\par\vspace{1ex}}
\newenvironment{AMS}
{\noindent{\bf AMS subject classifications 2010.}\small}{\par}

\newcommand{\DELETE}[1]{}

\newcommand{\fD}{\mathfrak{D}}

\title{The quadratic Wasserstein metric for inverse data matching}

\author{
	Bj\"orn Engquist\thanks{Department of Mathematics and the Oden Institute, The University of Texas, Austin, TX 78712;
		\href{mailto:engquist@ices.utexas.edu}{engquist@ices.utexas.edu}
	}
	\and
	Kui Ren\thanks{Department of Applied Physics and Applied Mathematics, Columbia University, New York, NY 10027;
		\href{mailto:kr2002@columbia.edu}{kr2002@columbia.edu}
	}
	\and
	Yunan Yang\thanks{Courant Institute of Mathematical Sciences, New York University, New York, NY 10012; 
		\href{mailto:yunan.yang@nyu.edu}{yunan.yang@nyu.edu}
	}
}

\begin{document}

\maketitle



\begin{abstract}
This work characterizes, analytically and numerically, two major effects of the quadratic Wasserstein ($W_2$) distance as the measure of data discrepancy in computational solutions of inverse problems. First, we show, in the infinite-dimensional setup, that the $W_2$ distance has a smoothing effect on the inversion process, making it robust against high-frequency noise in the data but leading to a reduced resolution for the reconstructed objects at a given noise level. Second, we demonstrate that, for some finite-dimensional problems, the $W_2$ distance leads to optimization problems that have better convexity than the classical $L^2$ and $\dot\cH^{-1}$ distances, making it a more preferred distance to use when solving such inverse matching problems.

\end{abstract}






\section{Introduction}
\label{SEC:Intro}

This paper is concerned with inverse problems where we intend to match certain given data to data predicted by a (discrete or continuous) mathematical model, often called the forward model. To set up the problem, we denote by a function $m(\bx): \bbR^d\to \bbR$ ($d\ge 1$) the input of the mathematical model that we are interested in reconstructing from a given datum $g$. We denote by $f$ the forward operator that maps the unknown quantity $m$ to the datum $g$, that is
\begin{equation}\label{EQ:Nonl IP}
	f(m)=g,
\end{equation}
where the operator $f$ is assumed to be nonlinear in general. We denote by $A:=f'(m_0)$ the linearization of the operator $f$ at the background $m_0$. With a bit of abuse of notation, we write $Am=g$ to denote a linear inverse problem of the form~\eqref{EQ:Nonl IP} where $f(m):=Am$. The space of functions where we take our unknown object $m$, denoted by $\cM$, and datum $g$, denoted by $\cG$, as well as the exact form of the forward operator $f: \cM \mapsto \cG$, will be given later when we study specific problems.


Inverse problems for~\eqref{EQ:Nonl IP} are mostly solved computationally due to the lack of analytic inversion formulas. Numerical methods often reformulate the problem as a data matching problem where one takes the solution as the function $m^*$ that minimizes the data discrepancy, measured in a metric $\fD$, between the model prediction $f(m)$ and the measured datum $g$. That is,
\begin{equation}\label{EQ:Min}
	m^* = \argmin_{m\in \cM}\Phi(m), \quad\mbox{with},\quad \Phi(m):= \dfrac{1}{2} \fD^2(f(m), g).
\end{equation}
The most popular metric used in the past to measure the prediction-data discrepancy is the $L^2$ metric $\fD(f(m), g):=\|f(m)-g\|_{L^2(\bbR^d)}$ due to its mathematical and computational simplicity. Moreover, it is often the case that a regularization term is added to the mismatch functional $\Phi(m)$ to impose extra prior knowledge on the unknown $m$ (besides of the fact that it is an element of $\cM$) to be reconstructed. 

In recent years, the quadratic Wasserstein metric~\cite{AmGi-MOFN13,Santambrogio-Book15,Villani-Book03} is proposed as an alternative for the $L^2$ metric in solving such inverse data matching problems~\cite{BrOpVi-Geophysics10,ChChWuYa-JCP18,EnFr-CMS14,KoPaThSlRo-IEEE17,LiViLeMa-IP19,LeLoScVa-SIAM14,MeBrMeOuVi-GJI16,MeBrMeOuVi-IP16,SoRuGuBu-PICML14,YaEn-Geophysics18}. Numerical experiments suggest that the quadratic Wasserstein metric has attractive properties for some inverse data matching problems that the classical $L^2$ metric does not have~\cite{EnFrYa-CMS16,YaEnSuFr-Geophysics18}. The objective of this work is trying to understand mathematically these numerical observations reported in the literature. More precisely, we attempt to characterize the numerical inversion  of~\eqref{EQ:Nonl IP} based on the quadratic Wasserstein metric and compare that with the inversion based on the classical $L^2$ metric.  

In the rest of the paper, we first briefly review some background materials on the quadratic Wasserstein metric and its connection to inverse data match problems in Section~\ref{SEC:Setup}. We then study in Section~\ref{SEC:Freq} the Fourier domain behavior of the solutions to~\eqref{EQ:Nonl IP} in the asymptotic regime of the Wasserstein metric: the regime where the model prediction $f$ and the datum $g$ are sufficiently close. We show that in the asymptotic regime, the Wasserstein inverse solution tends to be smoother than the $L^2$ based inverse solution. We then show in Section~\ref{SEC:W2 Iterations} that this smoothing effect of the Wasserstein metric also exists in the non-asymptotic regime, but in a less explicit way. In Section~\ref{SEC:W2 De-Conv}, we demonstrate, in perhaps overly simplified settings, some advantages of the Wasserstein metric over the $L^2$ metric in solving some important inverse matching problems: inverse transportation, back-projection from (possibly partial) data, and deconvolution of highly concentrated sources.  Numerical simulations are shown in Section~\ref{SEC:Num} to demonstrate the main findings of our study. Concluding remarks are offered in Section~\ref{SEC:Concl}.

\section{Background and problem setup}
\label{SEC:Setup}

Let $f$ and $g$ be two probability densities on $\bbR^d$ that have the same total mass. The square of the quadratic Wasserstein distance between $f$ and $g$ denoted as $W_2^2(f,g)$, is defined as
\begin{equation}\label{EQ:W2}
	W_2^2(f, g) := \inf_{\bT\in\cT} \int_{\bbR^d} |\bx-\bT(\bx)|^2 f(\bx) d \bx,
\end{equation}
where $\cT$ is the set of measure-preserving maps from $f$ to $g$. The map $\bT$ that achieves the infimum is called the optimal transport map between $f$ and $g$. In the context of~\eqref{EQ:Nonl IP}, the probability density $f$ depends on the unknown function $m$. Therefore, $W_2^2(f, g)$ can be viewed as a mismatch functional of $m$ for solving the inverse problem. 

Since the quadratic Wasserstein distance is only defined between probability measures of the same total mass, one has to normalize $f$ and $g$ and turn them into probability densities when applying them to solve inverse matching problems where $f$ and $g$ cannot be interpreted as nonnegative probability density functions. This introduces new issues that need to be addressed~\cite{EnYa-MAA19}; see more discussions at the end of Section~\eqref{SEC:Concl}. The analysis in the rest of the paper assumes that $f$ and $g$ are both nonnegative and have the same total mass.

It is well-known by now that the quadratic Wasserstein distance between $f$ and $g$ is connected to a weighted $\dot \cH^{-1}$ distance between them; see~\cite[Section 7.6]{Villani-Book03} and~\cite{Loeper-JMPA06, Peyre-ESAIM18}. For any $s\in\bbR$, let $\cH^s(\bbR^d)$ be the space of functions 
\[
	\cH^s(\bbR^d):=\{ m(\bx): \|m\|_{\cH^s(\bbR^d)}^2 := \int_{\bbR^d}\aver{\bxi}^{2s}|\wh m(\bxi)|^2d\bxi<+\infty\}
\]
where $\wh m(\bxi)$ denotes the Fourier transform of $m(\bx)$ and $\aver{\bxi}:=\sqrt{1+|\bxi|^2}$. When $s\ge 0$, $\cH^s(\bbR^d)$ is the usual Hilbert space of functions with $s$ square integrable derivatives, and $\cH^0(\bbR^d)=L^2(\bbR^d)$. The space $\cH^{-s}(\bbR^d)$ with $s>0$ is understood as the dual of $\cH^s(\bbR^d)$. We also introduce the space $\dot \cH^s(\bbR^d)$, $s>0$, with the (semi-) norm $\|\cdot\|_{\dot \cH^s(\bbR^d)}$ defined through the relation
\[
	\|m\|_{\cH^s(\bbR^d)}^2= \|m\|_{L^2(\bbR^d)}^2+\|m\|_{\dot \cH^s(\bbR^d)}^2.
\]
The space $\dot \cH^{-s}(\bbR^d)$ is defined as the dual of $\dot \cH^{s}(\bbR^d)$ via the norm
\begin{equation}
	\|m\|_{\dot \cH^{-s}}:=\sup\{|\aver{w, m}|: \|w\|_{\dot \cH^s}\le 1\}.
\end{equation}
It was shown~\cite[Section 7.6]{Villani-Book03} that \emph{asymptotically} $W_2$ is equivalent to $\dot \cH_{(d\mu)}^{-1}$, where the subscript $(d\mu)$ indicates that the space is defined with respect to the reference probability measure $d\mu=f(\bx) d\bx$. To be precise, if $\mu$ is the probability measure and $d\pi$ is an infinitesimal perturbation that has zero total mass, then
\begin{equation}\label{EQ:W2-Hm1 Asym}
	W_2(\mu, \mu+d\pi)=\|d\pi\|_{\dot \cH_{(d\mu)}^{-1}}+o(d\pi) .
\end{equation}
This fact allows one to show that, for two positive probability measures $\mu$ and $\nu$ with densities $f$ and $g$ that are sufficiently regular, we have the following \emph{non-asymptotic} equivalence between $W_2$ and $\dot\cH_{(d\mu)}^{-1}$:
\begin{equation}\label{EQ:W2-Hm1}
	\fc_1 \|\mu-\nu\|_{\dot \cH_{(d\mu)}^{-1}} \le W_2(\mu, \nu) \le \fc_2 \|\mu-\nu\|_{\dot \cH_{(d\mu)}^{-1}},
\end{equation}
for some constants $\fc_1>0$ and $\fc_2>0$. The second inequality is generic with $\fc_2=2$~\cite[Theorem 1]{Peyre-ESAIM18} while the first inequality, proved in~\cite[Proposition 2.8]{Loeper-JMPA06} and~\cite[Theorem 5]{Peyre-ESAIM18} independently, requires further that $f$ and $g$ be bounded from above.

In the rest of this paper, we study numerical solutions to the inverse data matching problem for~\eqref{EQ:Nonl IP} under three different mismatch functionals:
\begin{equation}\label{EQ:Hs Obj}
	\Phi_{\cH^s}(m)\equiv \frac{1}{2}\|f(m)-g\|^2_{\cH^s}:=\frac{1}{2}\int_{\bbR^d}\aver{\bxi}^{2s}|\wh f(m)(\bxi)-\wh g(\bxi)|^2 d\bxi,
\end{equation}
\begin{equation}\label{EQ:Weighted Hs Obj}
	\Phi_{\cH_{(d\mu)}^s}(m)\equiv \frac{1}{2}\|f(m)-g\|^2_{\cH_{(d\mu)}^s}:=\frac{1}{2}\int_{\bbR^d} \Big|\wh{\omega}\circledast \big[\aver{\bxi}^{s} (\wh f(m)(\bxi)-\wh g(\bxi))\big] \Big|^2 d\bxi,
\end{equation}
where $\omega(\bx)=1/g(\bx)$, $\circledast$ denotes the convolution operation, and
\begin{equation}\label{EQ:W2 Obj}
	\Phi_{W_2}(m)\equiv \frac{1}{2}W_2^2(f(m), g): =\dfrac{1}{2} \inf_{\bT\in\cT} \int_{\bbR^d} |\bx-\bT(\bx)|^2 f(m(\bx)) d \bx.
\end{equation}
Our main goal is to analyze the differences between the Fourier contents of the inverse matching results, a main motivation for us to choose the Fourier domain definition of the $\cH$ norms. These norms allow us to systematically study: (i) the differences between the $L^2$ (i.e. the special case of $s=0$ of $\Phi_{\cH^s}(m)$) and the $\cH^s$, with a positive or negative $s$, inversion results; (ii) the differences between $\cH^s$ and $\cH_{(d\mu)}^s$ inversion results caused by the weight $d\mu$; and (iii) the similarities and differences between $\cH_{(d\mu)}^s$ and $W_2$ inversion results. This is our roadmap toward better understandings of the differences between $L^2$-based and $W_2$-based inverse data matching.

\begin{remark}
	Note that since the $\dot\cH^s$ norm is only a shift away from the corresponding $\cH^s$ norm in the Fourier representation, by replacing $\aver{\bxi}$ with $|\bxi|$, we do not introduce extra mismatch functionals for those (semi-) norms. We will, however, discuss $\dot\cH^s$ inversions when we study the corresponding $\cH^s$ inversions.
\end{remark}

\begin{remark}
	In the definition of the $\cH_{(d\mu)}^s$ objective function, we take the weight function $\omega(\bx)=1/g(\bx)$ such that $\|f-g\|^2_{\dot\cH^{-1}_{(d\mu)}}$ is consistent with the linearization of $W_2^2(f, g)$~\cite{Villani-Book03}.
\end{remark}

We refer interested readers to~\cite{Villani-Book03,Loeper-JMPA06, Peyre-ESAIM18} for technical discussions on the results in~\eqref{EQ:W2-Hm1 Asym} and ~\eqref{EQ:W2-Hm1} (under more general settings than what we present here) that connect $W_2$ with $\dot \cH_{(d\mu)}^{-1}$. For our purpose, these results say that: (i) in the asymptotic regime when two signals $f$ and $g$, both being probability density functions, are sufficiently close to each other, their $W_2$ distance can be well approximated by their $\dot \cH_{(d\mu)}^{-1}$ distance; and (ii) if $W_2(f, g)=0$, then $\|f-g\|_{\dot\cH_{(d\mu)}^{-1}}=0$ and vice versa, that is, the exact matching solutions to the model~\eqref{EQ:Nonl IP}, if exists, are global minimizers to both $\Phi_{W_2}(m)$ and $\Phi_{\cH^{-1}_{(d\mu)}}(m)$. \emph{However, let us emphasize that the non-asymptotic equivalence in~\eqref{EQ:W2-Hm1} does NOT imply that the functional $\Phi_{W_2}(m)$ and $\Phi_{\dot \cH^{-1}_{(d\mu)}}(m)$ (if we define one) have exactly the same optimization landscape}. In fact, numerical evidences show that the two functionals have different optimization landscapes that are both quite different from that of the $L^2$ mismatch functional $\Phi_{L^2}(m):=\Phi_{\cH^0}(m)$; see for instance Section~\eqref{SEC:W2 De-Conv} for analytical and numerical evidences.

\section{Frequency responses in asymptotic regime}
\label{SEC:Freq}

We first study the Fourier-domain behavior of the solutions to~\eqref{EQ:Nonl IP} obtained by minimizing the functionals we introduced in the previous section. At the solution, $f(m)$ and $g$ are sufficiently close to each other. Therefore their $W_2$ distance can be replaced with their $\dot\cH_{(d\mu)}^{-1}$ distance according to~\eqref{EQ:W2-Hm1 Asym}. In the leading order, the $W_2$ solution is simply the $\dot\cH_{(d\mu)}^{-1}$ solution in this regime.

\subsection{$\cH^{s}$-based inverse matching for linear problems}

Let us start with a linear inverse matching problem given by the model:
\begin{equation} \label{EQ:Lin IP}
	Am =g_\delta,
\end{equation}
where $g_\delta$ denotes the datum $g$ in~\eqref{EQ:Nonl IP} polluted by an additive noise introduced in the measuring process. The subscript $\delta$ is used to denote the size  (in appropriate norms to be specified soon) of the noise, that is, the size of $g_\delta -g$. Besides, we assume that $g_\delta$ is still in the range of the operator $A$. When the model is viewed as the linearization of the nonlinear model~\eqref{EQ:Nonl IP}, $m$ should be regarded as the perturbation of the background $m_0$. The model perturbation is also often denoted as $\delta m$.  We assume that the linear operator $A$ is diagonal in the Fourier domain, that is, it has the symbol,
\begin{equation}\label{EQ:IP Symbol}
	\wh A(\bxi)\sim \aver{\bxi}^{-\alpha},
\end{equation}
for some $\alpha\in\bbR$. \emph{This assumption is to make some of the calculations more concise but is not essential as we will comment on later; see Remark~\ref{RMK:Diagonal}}. When the exponent $\alpha> 0$, the operator $A$ is ``smoothing'', in the sense that it maps a given $m$ to an output with better regularity than $m$ itself. The inverse matching problem of solving for $m$ in~\eqref{EQ:Lin IP}, on the other hand, is ill-conditioned (so would be the corresponding nonlinear inverse problem $f(m)=g$ if $A$ is regarded as the linearization of $f$). The size of $\alpha$, to some extent, can describe the degree of ill-conditionedness of the inverse matching problem. 

We assume \emph{a priori} that $m\in \cH^\beta(\bbR^d)$ for some $\beta>0$. Therefore, $A$ could be viewed as an operator $A: \cH^\beta \mapsto \cH^{\beta+\alpha}$. We now look at the inversion of the problem under the $\cH^s$ ($s\le \alpha+\beta$) framework. 

We seek the solution of the inverse problem as the minimizer of the $\cH^s$ functional $\Phi_{\cH^s}(m)$, defined as in~\eqref{EQ:Hs Obj} with $f(m)=Am$ and $g$ replaced with $g_\delta$. We verify that the Fr\'echet derivative of $\Phi_{\cH^s}: \cH^\beta \mapsto \bbR_{\ge 0}$ at $m$ in the direction $\delta m$ is given by
\[
	\Phi_{\cH^s}'(m)[\delta m]= \int_{\bbR^d}  \wh A^*(\bxi)\left\{\aver{\bxi}^{2s} \overline{\left[\wh A(\bxi) \wh m(\bxi)-\wh{g_\delta}(\bxi)\right]}\right\} \wh {\delta m}(\bxi)  d\bxi,
\]
where we used $A^*$ to denote the $L^2$ adjoint of the operator $A$. The minimizer of $\Phi_{\cH^s}$ is located at the place where its Fr\'echet derivative vanishes. Therefore the minimizer solves the following (modified) normal equation at frequency $\bxi$:
\begin{equation}\label{EQ:NE Hs}
	\wh A^*(\bxi) \left\{\aver{\bxi}^{2s} \wh A\right\} \wh m = \wh A^*(\bxi) \left\{\aver{\bxi}^{2s} \wh{g_\delta}(\bxi) \right\}.
\end{equation}
The solution at frequency $\bxi$ is therefore
\[
	\wh m(\bxi) =\Big(\wh A^*(\bxi) \big(\aver{\bxi}^{2s}  \wh A\big)\Big)^{-1} \wh A^*(\bxi) \Big(\aver{\bxi}^{2s}\wh{g_\delta}(\bxi)\Big).
\]
We can then perform an inverse Fourier transform to find the solution in the physical space. The result is
\begin{equation}\label{EQ:Hs Inversion Phy}
	m =\Big(A^* P A \Big)^{-1} A^* P g_\delta, \qquad P:= (\cI-\Delta)^{s/2},
\end{equation}
where the operator $(\cI-\Delta)^{s/2}$ is defined through the relation 
\[
	(\cI-\Delta)^{s/2}m:=\cF^{-1} \Big(\aver{\bxi}^s\wh m\Big),
\]
$\cF^{-1}$ being the inverse Fourier transform, $\Delta$ being the usual Laplacian operator, and $\cI$ being the identity operator.

\paragraph{Key observations.} Let us first remark that the calculations above can be carried out in the same way if the $\cH^s$ norm is replaced with the $\dot\cH^s$ norm. The only changes are that $\aver{\bxi}$ should be replaced with $|\bxi|$ and the operator $\cI-\Delta$ in $P$ has to be replaced with $-\Delta$.

When $s=0$, assuming that $\alpha+\beta\ge 0$, the above solution reduces to the classical $L^2$ least-squares solution $m=(A^*A)^{-1}A^*g_\delta$. Moreover, when $A$ is invertible (so will be $A^*$), the solution can be simplified to $m=A^{-1} g_\delta$, using the fact that $\big(A^* P  A \big)^{-1}=A^{-1} P^{-1} A^{-*}$, which is simply the true solution to the original problem~\eqref{EQ:Lin IP}. Therefore, in the same manner, as the classical $L^2$ least-squares method, the least-squares method based on the $\cH^s$ norm does not change the solution to the original inverse problem when it is uniquely solvable. This is, however, not the case for the $\dot\cH^s$ inversion in general. For instance, $\dot\cH^1$ inversion only matches the derivative of the predicted data to the measured data.

When $s>0$, $P$ is a differential operator. Applying  $P$ to the datum $g_\delta$ amplifies high-frequency components of the datum.  When $s<0$,  $P$ is a (smoothing) integral operator. Applying  $P$ to the datum $g_\delta$ suppresses high-frequency components of the datum. Even though the presence of the operator $P$ in $\Big(A^* P A \Big)^{-1}$ will un-do the effect of  $P$ on the datum in a perfect world (when $A$ is invertible, and all calculations are done with arbitrary precision), when operated under a given accuracy, inversion with $s<0$ is less sensitive to high-frequency noise in the data while inversion with $s>0$ is more sensitive to high-frequency noise in the data, compared to the case of $s=0$ (that is, the classical $L^2$ least-squares inversion). Therefore, inversion with $s\neq 0$ can be seen as a ``preconditioned'' (by the operator $P$) $L^2$ least-squares inversion.

\subsection{Resolution analysis of linear inverse matching}

We now perform a simple resolution analysis, following the calculations in~\cite{BaRe-IP09}, on the $\cH^s$ inverse matching result for the linear model~\eqref{EQ:Lin IP}.
\begin{theorem}\label{THM:Error Analysis}
	Let $A$ be given as in~\eqref{EQ:IP Symbol} and $R_c$ an approximation to $A^{-1}$ defined through its symbol:
\begin{equation*}
  \wh R_c(\bxi) \sim \left\{
    \begin{array}{ll}
         \aver{\bxi}^\alpha, \qquad & |\bxi|<\xi_c \\
         0, & |\bxi|>\xi_c
    \end{array} \right..
\end{equation*}
Let $\delta=\|g_\delta-g\|_{\cH^s}$ be the $\cH^s$ norm of the additive noise in $g_\delta$. Then the reconstruction error $\|m-m_\delta^c\|_{L^2}$, with $m_\delta^c:=R_c g_\delta$ obtained as the minimizer of $\Phi(m)_{\cH^s}$, is bounded by
\begin{equation} \label{EQ:Error LIP}
  \|m-m^c_\delta\|_{L^2} \lesssim \|m\|_{\cH^\beta}^{\frac{\alpha-s}{\alpha+\beta-s}} \delta^{\frac\beta{\alpha+\beta-s}}.
\end{equation}
This optimal bound is achieve when we select
\begin{equation}\label{EQ:xic}
	\aver{\xi_c}^{-1} \sim (\delta\|m\|_{\cH^\beta}^{-1})^{\frac{1}{\alpha+\beta-s}}.
\end{equation}
\end{theorem}
\begin{proof}
Following classical results in~\cite{EnHaNe-Book96}, it is straightforward to verify that the $L^2$ difference between the true solution $m$ and the approximated noisy solution $m^c_\delta$ is
\begin{multline}\label{EQ:Bound}
  \|m-m^c_\delta\|_{L^2} =\|m-R_c g_\delta\|_{L^2} =\|m-R_cg+R_cg-R_c g_\delta\|_{L^2} \\ = \|(\cI-R_cA)m+R_c(g-g_\delta)\|_{L^2} \le \|(\cI-R_cA)m\|_{L^2}+\|R_c(g-g_\delta)\|_{L^2}.
\end{multline}
We then verify that operators $\cI-R_cA: \cH^\beta(\bbR^d) \mapsto L^2(\bbR^d)$ and $R_c: \cH^s(\bbR^d)\mapsto L^2(\bbR^d)$ have the following norms respectively
\[
	\|R_c\|_{\cH^s\mapsto L^2}\sim \aver{\xi_c}^{\alpha-s} \qquad \mbox{and}\qquad \|(\cI-R_c A)\|_{\cH^\beta \mapsto L^2}\sim \aver{\xi_c}^{-\beta}.
\]
This allows us to conclude that
\begin{multline*}\label{EQ:Bound}
  \|m-m^c_\delta\|_{L^2} \leq \|R_c\|_{\cH^s\mapsto L^2} \delta + \|(\cI-R_c A)\|_{\cH^\beta \mapsto L^2}\|m\|_{\cH^\beta} 
  \lesssim \aver{\xi_c}^{\alpha-s} \delta + \aver{\xi_c}^{-\beta}\|m\|_{\cH^\beta}.
\end{multline*}
We can now select $\aver{\xi_c}\sim (\delta^{-1}\|m\|_{\cH^\beta})^{\frac{1}{\alpha+\beta-s}}$, i.e. the relation given in~\eqref{EQ:xic}, to minimize the error of the reconstruction, which gives the bound in~\eqref{EQ:Error LIP}.
\end{proof}

\paragraph{Optimal resolution.} One main message carried by the theorem above is that reconstruction based on the $\cH^s$ mismatch has a spatial resolution
\[
	\eps:=\aver{\xi_c}^{-1}\sim\delta^{\frac{1}{\alpha+\beta-s}},
\]
under the conditions in the theorem. At a fixed noise level $\delta$, for fixed $\alpha$ and $\beta$, the optimal resolution of the inverse matching result degenerates when $s$ gets smaller. The case of $s=0$ corresponds to the usual reconstruction in the $L^2$ framework. The optimal resolution one could get in this case is decided by $\delta^{\frac{1}{\alpha+\beta}}$. When $0<s$ ($\le \alpha+\beta$), the best resolution one could get is better than the $L^2$ case in a perfect world.  When $s<0$, the reconstructions in the $\cH^s$ framework provides an optimal resolution that is worse than the $L^2$ case. In other words, the reconstructions based on the negative norms appear to be smoother than optimal $L^2$ reconstructions in this case. See Section~\ref{SEC:Num} for numerical examples that illustrate this resolution analysis.

However, we should emphasize that the above simple calculation only provides the best-case scenarios. It does not tell us exactly how to achieve the best results in a general setup (when the symbol of $A$, i.e., the singular value decomposition of $A$ in the discrete case, is not precisely known). Nevertheless, the guiding principle of the analysis is well demonstrated: least-squares with a stronger (than the $L^2$) norm yield higher resolution reconstructions while least-squares with a weaker norm yield lower (again compared to the $L^2$ case) resolution reconstructions in the best case.


\subsection{$\cH_{(d\mu)}^{s}$-based inverse matching for linear problems}

Inverse matching with the weighted $\cH^s$ norm can be analyzed in the same manner to study the impact of the weight on the inverse matching result. The solution $m$ to~\eqref{EQ:Lin IP} in this case is sought as the minimizer of the functional $\Phi_{\cH^s_{(d\mu)}}(m)$ defined in~\eqref{EQ:Weighted Hs Obj} with $f(m)=Am$ and $g=g_\delta$. This means that the weight $\omega=1/g_\delta$ in our definition of the objective function.

Following the same calculation as in the previous subsection, we find that the minimizer of the functional $\Phi_{\cH^s_{(d\mu)}}(m)$ solves the following normal equation at frequency $\bxi$:
\begin{equation}\label{EQ:NE Hs Weighted}
	\wh B^* \wh{\omega}\circledast \big(\aver{\bxi}^{s} \wh A\wh m\big) =\wh B^* \wh{\omega}\circledast \big(\aver{\bxi}^{s} \wh{g_\delta} \big)
\end{equation}
where $\wh B^*$ is the $L^2$ adjoint of the operator $\wh B$ defined through the relation $\wh B \hat m: = \wh{\omega}\circledast \big(\aver{\bxi}^{s} \wh A \wh m \big)$. 

We first observe that the right-hand side of ~\eqref{EQ:NE Hs Weighted} and that of ~\eqref{EQ:NE Hs} are different. In~\eqref{EQ:NE Hs}, the $\bxi$-th Fourier mode of the datum $g_\delta$ is amplified or attenuated, depending on the sign of $s$, by a factor of $\aver{\bxi}^{2s}$. While in~\eqref{EQ:NE Hs Weighted}, this mode is further convoluted with other modes of the datum after the amplification/attenuation. The convolution induces mixing between different modes of the datum. Therefore, inverse matching with the weighted norm cannot be done mode by mode as what we did for the unweighted norm, even when we assume that the forward operator $A$ is diagonal. However, main effect of the norm on the inversion, the smoothing/sharpening effect introduced by the $\aver{\bxi}^{2s}$ factor (half of which come from the factor  $\aver{\bxi}^{s}$ in front of $\wh{g_\delta}$ while the other half come from the factor  $\aver{\bxi}^{s}$ in $\wh B$), are the same in both the unweighted $\cH^s$ and the weighted $\cH_{(d\mu)}^s$ norms.

The inverse matching solution, in this case, written in physical space, is:
\begin{equation}\label{EQ:Hs Weighted Inversion Phy}
	m =\Big(A^* P_g  A \Big)^{-1} A^* P_g g_\delta, \qquad P_g := (\cI-\Delta)^{s/2} \omega  (\cI-\Delta)^{s/2}.
\end{equation}
We can again compare this with the unweighted solution in ~\eqref{EQ:Hs Inversion Phy}. The only difference is the introduction of the inhomogeneity, which depends on the datum $g_\delta$, in the preconditioning operator $P$ by replacing it with $P_g$. When $0<s$ ($\le \alpha+\beta$),  $P$ and $P_g$  are (local) differential operators. Roughly speaking, compared to $P$, $P_g$ emphasizes places where $g_\delta$ is small, be reminded that $\omega=1/g_\delta$, or the $s$-th order derivative of $g_\delta$ is large. At those locations, $P_g$ amplifies the same modes of the datum $g_\delta$ more than $P$ does. When $s<0$, $P$ and $P_g$ are non-local operators. The impact of $g_\delta$ is more global (as we have seen in the previous paragraph in the Fourier domain). It is hard to precisely characterize the impact of $g_\delta$ without knowing its form explicitly. However, we can still see, for instance, from~\eqref{EQ:NE Hs Weighted}, that the smoother $g_\delta$ is, the smoother the inverse matching result will be (since $\wh{g_\delta}$ has fast decay and the convolution will further smooth out $\aver{\bxi}^s g_\delta$). If $g_\delta$ is very rough, say that it behaves like Gaussian noise, then $\wh{g_\delta}$ decays very slowly. The convolution, in this case, will not smooth out $\aver{\bxi}^s g_\delta$ as much as in the previous case. The main effect of $\cH_{(d\mu)}^s$ on the inverse matching result in this case mainly comes from the norm, not the weight.

\begin{remark}
Thanks to the asymptotic equivalence between $\dot\cH_{(d\mu)}^{-1}$ and $W_2$ in~\eqref{EQ:W2-Hm1 Asym}, the smoothing effect we observe in this section for the $\cH_{(d\mu)}^{-1}$ inverse matching (and therefore $\dot\cH_{(d\mu)}^{-1}$ inverse matching since $\dot\cH_{(d\mu)}^{-1}$ is only different from $\cH_{(d\mu)}^{-1}$ on the zeroth moment in the Fourier domain) is also what we observe in the $W_2$ inverse matching. This observation will be demonstrated in more detail in our numerical simulations in Section~\ref{SEC:Num}.
\end{remark}

\subsection{Iterative solution of nonlinear inverse matching}

The simple analysis in the previous sections based on the linearized quadratic Wasserstein metric, i.e., a weighted $\dot\cH^{-1}$ norm, on the inverse matching of linear model~\eqref{EQ:Lin IP} does not translate directly to the case of inverse matching with the fully nonlinear model~\eqref{EQ:Nonl IP}. Nevertheless, the analysis does provide us some insights.

Let us consider an iterative matching algorithm for the nonlinear problem, starting with a given initial guess $m_0$, characterized by the following iteration: 
\begin{equation}\label{EQ:Iteration}
	m_{k+1}=m_{k}+\ell_k\, \zeta_k,\ \ k\ge 0,
\end{equation}
where $\zeta_k$ is a chosen descent direction of the objective functional at iteration $k$, and $\ell_k$ is the step length at this iteration. For simplicity, let us take the steepest descent method where the descent direction is taken as the negative gradient direction. Following the calculations in the previous section, we verify that the Fr\'echet derivative of $\Phi_{\cH_{(d\mu)}^{s}}(m): \cH^\beta \mapsto \bbR_{\ge 0}$ at the current iteration $m_k$ in the direction $\delta m$ is given by
\begin{equation}\label{EQ:Hs Weighted Frechet}
	\Phi'_{\cH^s_{(d\mu)}}(m_k)[\delta m] =\int_{\bbR^d} \wh{\omega}\circledast \big[\aver{\bxi}^{s} (\wh{f(m_k)}-\wh g_\delta)\big] \wh{\omega}\circledast \aver{\bxi}^{s} \wh{f'(m_k)[\delta m]}  d\bxi,
\end{equation}
assuming that the forward model $f: \cH^\beta \mapsto \cH_{(d\mu)}^s$ is Fr\'echet differentiable at $m_k$ with derivative $f'(m_k)[\delta m]$. This leads to the following descent direction $\zeta_k$ chosen by a gradient descent method: 
\begin{equation}\label{EQ:Hs Weighted Gradient}
	\zeta_k =-\Big(f'(m_k)^* P_g  f'(m_k) \Big)^{-1} f'(m_k)^*P_g (f(m_k)-g_\delta),
\end{equation}
Let us compare this with the descent direction resulted from the $L^2$ least-squares functional:
\begin{equation}\label{EQ:L2 Gradient}
	\zeta_k=\Big(f'(m_k)^* f'(m_k) \Big)^{-1} f'(m_k)^* (f(m_k)-g).
\end{equation}
We see that the iterative process of the $\cH_{(d\mu)}^{s}$ inverse matching can be viewed as a preconditioned version of the corresponding $L^2$ iteration. The preconditioning operator, $P_g$, depends on the datum $g_\delta$ but is independent of the iteration. When the iteration is stopped after a finite step, the effect we observed for linear problems, that is, the smoothing effect of $P_g$ in the case of $s<0$ or its de-smoothing effect in the case of $s>0$,  is carried to the solution of nonlinear problems.

\paragraph{Wasserstein smoothing in the asymptotic regime.} To summarize, when the model predictions and the measured data are sufficiently close to each other, inverse matching with the quadratic Wasserstein metric, or equivalently the $\dot\cH_{(d\mu)}^{-1}$ metric, can be viewed as a preconditioned $L^2$-based inverse matching. The preconditioning operator is roughly the inverse Laplacian operator with a coefficient given by the datum. The optimal resolution of the inversion result from the Wasserstein metric, with data at a given noise level $\delta$ is roughly of the order $\delta^{\frac{1}{\alpha+\beta+1}}$ ($\alpha$ being the order of the operator $f'(m)$ at the optimal solution and $m\in\cH^\beta$) instead of $\delta^{\frac{1}{\alpha+\beta}}$ as given in the $L^2$ case. The shape of the datum $g_\delta$ distorts the Wasserstein matching result slightly from the inverse matching result with a regular $\dot\cH^{-1}$ (semi-) norm.



\begin{remark}\label{RMK:Diagonal}
The assumption that the linear operator $A$ is diagonal in the Fourier domain, given in ~\eqref{EQ:IP Symbol}, simplifies the calculations in this section. The assumption is not necessary at all to show the preconditioning effect of the $W_2$ metric. Without this assumption, we need to replace all the multiplication of $\wh A$ in the Fourier domain with convolutions. The final results remain the same. The assumption is indeed necessary in order to write down the approximate inverse operator $R_c$ in the Fourier domain explicitly. This leads to a precise resolution characterization in Theorem~\ref{THM:Error Analysis} for this reconstruction operator.
\end{remark}
\section{Wasserstein iterations in non-asymptotic regime}
\label{SEC:W2 Iterations}

As we have seen from the previous sections, in the asymptotic regime, the smoothing effect of the quadratic Wasserstein metric in solving inverse matching problems can be characterized relatively precise thanks to the equivalence between $W_2$ and $\dot\cH_{(d\mu)}^{-1}$ given in~\eqref{EQ:W2-Hm1 Asym}. The demonstrated smoothing effect makes $W_2$-based inverse matching very robust against high-frequency noise in the measured data. This phenomenon has been reported in the numerical results published in recent years~\cite{BrOpVi-Geophysics10,EnFr-CMS14,EnFrYa-CMS16,YaEn-Geophysics18,YaEnSuFr-Geophysics18} and is one of the main reasons that $W_2$ is considered as a good alternative for $L^2$-based matching methods. In this section, we argue that the smoothing effect of $W_2$ can also be observed in the non-asymptotic regime, that is, a regime where signals $f$ and $g$ are sufficiently far away from each other. \emph{The smoothing effect in the non-asymptotic regime implies that the landscape of the $W_2$ objective functional is smoother than that of the classical $L^2$ objective functional.}

To see the smoothing effect of $W_2$ in non-asymptotic regime, we analyze the inverse matching procedure described by the iterative scheme~\eqref{EQ:Iteration} for the objective functional $\Phi_{W_2}(m)$, defined in~\eqref{EQ:W2 Obj}. For the sake of being technically correct, we assume that the data we are comparing in this section are sufficiently regular. More precisely, we assume that $f\in \cC^{0,\alpha}(\bbR^d)$ and $g\in \cC^{0,\alpha}(\bbR^d)$ for some $\alpha>0$. We also assume that the map $m \mapsto f$ ($\cH^\beta\to \cC^{0,\alpha}$) is Fr\'echet differentiable at any admissible $m$ and denote by $f'(m)[\delta m]$ the derivative in direction $\delta m$. We can then write down the variation of $\Phi_{W_2}(m):\ \cH^\beta \mapsto \bbR_{\ge 0}$ at the current iteration $m_k$ in the direction $\delta m$, following the differentiability result of $W_2^2(f,g)$ with respect to $f$ along mass preserving perturbations~\cite[Theorem 8.13]{Villani-Book03}. It is, 
\begin{multline}\label{EQ:W2 Gradient}
	\Phi_{W_2}'(m_k)[\delta m] =\int_{\bbR^d} \Big(\dfrac{|\bx-\bT_k(\bx)|^2}{2}f'(m_k)[\delta m] - (\bx-\bT_k(\bx))f(\bx) \cdot \bT_k'\big[f'(m_k)[\delta m]\big]\Big) d\bx,
\end{multline}
where $\bT_k$ denotes the optimal transport map at iteration $k$ (that is, for $m_k$), and $\bT_k'[\delta f]$ denotes the variation of $\bT_k$ with respect to $f$ (not $m$) in the direction $\delta f$. We emphasize again that $\delta m$ is selected such that $\int_{\bbR^d}f(m_k)d\bx=\int_{\bbR^d} f(m_k+\delta m) d\bx$ which is necessary since the space of probability densities with the $W_2$ metric is not a linear space.

Following the optimal transport theorem of Brenier~\cite{Villani-Book03}, the optimal transport map at the current iteration $k$, $\bT_k$, is given as $\bT_k(\bx): = \nabla u(\bx)$ where $u$ is the unique (up to a constant) convex solution to the Monge-Amp\`ere equation:
\begin{equation}\label{EQ:Monge-Ampere}
	\det(D^2 u(\bx)) = f(m_k(\bx))/g(\nabla u(\bx)),\ \ u\ \mbox{being convex}.
\end{equation}
Here $D^2$ is the Hessian operator defined through the Hessian matrix $D^2 u:=(\dfrac{\partial^2 u}{\partial x_i \partial x_j})$ (with the notation $\bx=(x_1, \cdots, x_d)$). Interested readers are referred to~\cite{Villani-Book03} and the references therein for more detailed mathematical study on the existence and uniqueness of solutions to the Monge-Amp\`ere equation~\eqref{EQ:Monge-Ampere}.

Let $\varphi:=u'(f(m_k))[\delta f]$ be the Fr\'echet derivative of $u$ at $f(m_k)$ in the direction $\delta f$, we then verify that $\varphi$ solves the following second-order elliptic equation to the leading order:
\begin{equation}\label{EQ:Monge-Ampere Lin}
	\sum_{ij}a_{ij}\dfrac{\partial^2 \varphi}{\partial x_i \partial x_j}  + \sum_{j} b_j\dfrac{\partial \varphi}{\partial x_j}  = \delta f,
\end{equation}
where $b_j=\det(D^2 u) \partial_{x_j} g(\bT_k(\bx))$ while $a_{ij}$ depend on the dimension. When $d=2$, $a_{ij}=-g(\bT_k(\bx))\dfrac{\partial^2 u}{\partial x_j \partial x_i}$ ($i\neq j$) and $a_{ii}=g(\bT_k(\bx))\dfrac{\partial^2 u}{\partial x_j \partial x_j}$ ($i\neq j$). When $d=3$, we have
\[
	a_{ij}=g(\bT_k(\bx))\left\{
	\begin{array}{rl}
		\dfrac{\partial^2 u}{\partial x_k \partial x_i}\dfrac{\partial^2 u}{\partial x_j \partial x_k}-\dfrac{\partial^2 u}{\partial x_j \partial x_i}\dfrac{\partial^2 u}{\partial x_k \partial x_k}, &  i\neq j \neq k,\\
		\dfrac{\partial^2 u}{\partial x_{k'} \partial x_{k'}}\dfrac{\partial^2 u}{\partial x_k \partial x_k}-\dfrac{\partial^2 u}{\partial x_{k'} \partial x_k}\dfrac{\partial^2 u}{\partial x_k \partial x_{k'}}, &  i=j\neq k \neq k'.
	\end{array} \right.
\]

Let $\psi$ be the solution to the (adjoint) equation:
\begin{equation}\label{EQ:Monge-Ampere Adj}
	\sum_{ij}a_{ij} \dfrac{\partial^2 \psi}{\partial x_i \partial x_j} - \sum_{j} b_j\dfrac{\partial \psi}{\partial x_j}  = -\nabla\cdot \Big(\big(\bx-\bT_k(\bx)\big)f(\bx)\Big).
\end{equation}
It is then straightforward to verify, following standard adjoint state calculations~\cite{Vogel-Book02}, that update direction can be written as
\begin{equation}\label{EQ:Direction}
	\zeta_k(\bx) = f'^{*}(m_k)\Big[\dfrac{|\bx-\bT_k(\bx)|^2}{2} + \psi(\bx) \Big],
\end{equation}
where $f'^*(m_k)$ denotes the $L^2$ adjoint of the operator $f'(m_k)$. 

We first observe that unlike in the classical $L^2$ case where $f'^*(m_k)$ is applied directly to the residual $f(m_k)-g$, that is, $\zeta_k(\bx)= f'^{*}(m_k)\Big[f(m_k)-g \Big]$, the descent direction here depends on the model prediction $f(m_k)$ and the datum $g$ only implicitly through the transfer map $\bI-\bT_k$ and its variation with respect to $m$. Only in the asymptotic regime of $g$ being very close to $f$ can we make the connection between $\dfrac{|\bx-\bT_k(\bx)|^2}{2} + \psi(\bx)$ and the normalized residual. This is where the $\dot\cH_{(d\mu)}^{-1}$ approximation to $W_2$ comes from. 

From Caffarelli's regularity theory (c.g. ~\cite[Theorem 4.14]{Villani-Book03}), which states that when $f\in \cC^{0,\alpha}(\bbR^d)$ and $g\in \cC^{0,\alpha}(\bbR^d)$ we have that the Monge-Amp\`ere solution $u\in\cC^{2,\alpha}(\bbR^d)$, we see that $(\bx-\bT_k(\bx))$ is at least $\cC^{1,\alpha}$. Therefore the solution to the adjoint problem, $\psi$, is in $\cC^{2,\alpha}$ by standard theory for elliptic partial differential equations when the problem is not degenerate and in $\cC^{1,\alpha}$ if it is degenerate. Therefore,  $\frac{|\bx-\bT_k(\bx)|^2}{2} + \psi(\bx)\in \cC^{1, \alpha}$ is one derivative smoother than $f$ and $g$ (and therefore the residual). This is exactly what the preconditioning operator $P$ (with $s=-1$) did to the residual in the asymptotic regime, for instance, as shown in~\eqref{EQ:Hs Inversion Phy}. This shows that $W_2$ inverse matching has smoothing effect even in the non-asymptotic regime.

In one-dimensional case, we can see the smoothing effect more explicitly since we are allowed to construct the optimal map explicitly in this case. Let $F$ and $G$ be the cumulative density functions for $f$ and $g$ respectively. The optimal transportation theorem in one-dimensional setting (c.g. ~\cite[Theorem 2.18]{Villani-Book03}) then says that the optimal transportation map from $f$ to $g$ is given by $T(x)=G^{-1}\circ F(x)$. This allows us to verify that, the gradient of $\Phi_{W_2}(m)$ at $m_k$ in direction $\delta m$, given in~\eqref{EQ:W2 Gradient}, is simplified to: 
\begin{multline}\label{EQ:W2 Gradient 1D}
	\Phi_{W_2}'(m_k)[\delta m] = \int_\bbR \Big(\dfrac{(x-T_k(x))^2}{2} + p_k(+\infty)-p_k(x)\Big)f'(m_k)[\delta m] dx \\
		= \int_\bbR f'^*(m_k)\Big[\dfrac{(x-T_k(x))^2}{2} - p_k(+\infty) + p_k(x)\Big] \delta m(x) dx
\end{multline}
where the function $p_k(x)$ is defined as $p_k(x)=\dint_{-\infty}^x \frac{(y-T_k(y))f(m_k(y))}{g(T_k(y))} dy$. Therefore the descent direction~\eqref{EQ:Direction} simplifies to
\begin{equation}\label{EQ:Direction 1D}
	\zeta_k(\bx) = f'^*(m_k)\Big[\dfrac{(x-T_k(x))^2}{2} - p_k(+\infty) + p_k(x)\Big].
\end{equation}

It is clear from~\eqref{EQ:Direction 1D} that the gradient of $\Phi_{W_2}(m)$ at iteration $k$ depends only on the anti-derivatives of $f(m_k)$, $g$ and $f(m_k)/g(T_k)$, through $T_k(x)$ and $p_k$. Therefore, it is smoother than the Fr\'echet derivative of $\Phi_{L^2}(m)$ in general, whether or not the signals $f(m_k)$ and $g$ are close to each other. This shows why the smoothing effect of $W_2$ exists also in non-asymptotic regime.
\begin{figure}[htb!]
	\centering
	\includegraphics[width=0.3\linewidth]{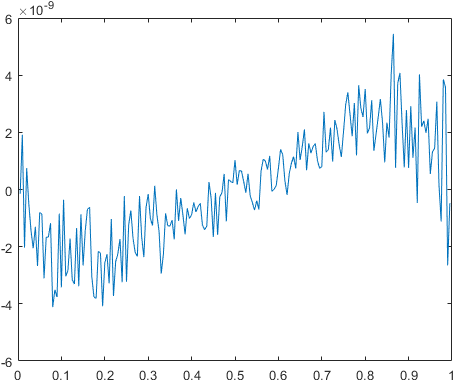}
	\includegraphics[width=0.3\linewidth]{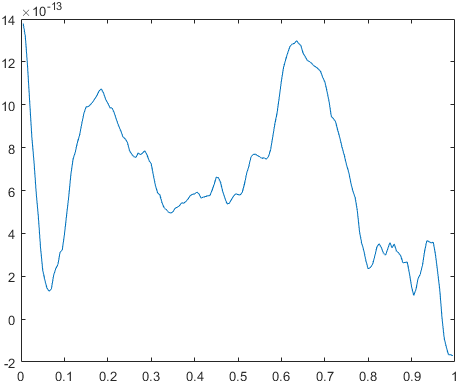}
	\includegraphics[width=0.3\linewidth]{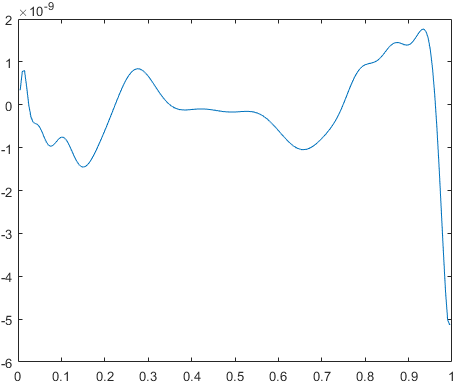}
	\caption{The gradients of the objective functions $\Phi_{L^2}(m)$  (left), $\Phi_{\cdot\cH^{-1}}(m)$ (middle) and $\Phi_{W_2}(m)$ (right) at the initial guess for the inverse diffusion problem in Section~\ref{SUBSEC:PAT} in the one-dimensional domain $\Omega=(0, 1)$.}
	\label{FIG:Frechet Derivative PAT 1D}
\end{figure}

\begin{figure}[htb!]
	\centering
	\includegraphics[width=0.3\linewidth]{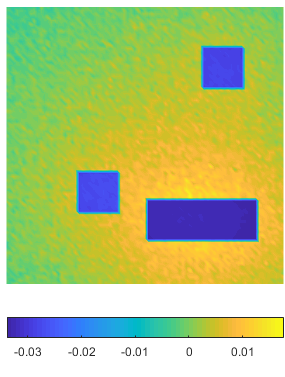}
	\includegraphics[width=0.3\linewidth]{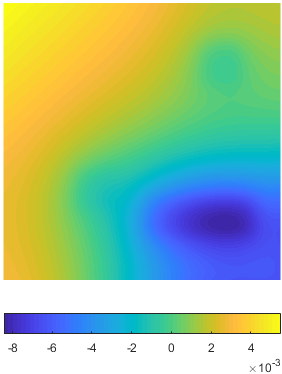}
	\includegraphics[width=0.3\linewidth]{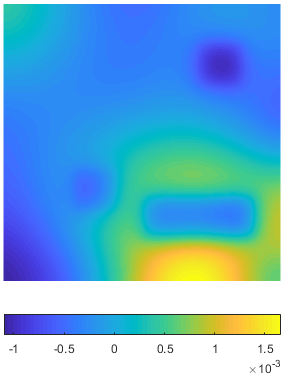}
	\caption{The gradients of the objective functions $\Phi_{L^2}(m)$  (left), $\Phi_{\cdot\cH^{-1}}(m)$ (middle) and $\Phi_{W_2}(m)$ (right) at the initial guess for the inverse diffusion problem in Section~\ref{SUBSEC:PAT} in the two-dimensional domain $\Omega=(0, 1)\times(0, 1)$.}
	\label{FIG:Frechet Derivative PAT 2D}
\end{figure}
To provide some numerical evidences, we show in Figure~\ref{FIG:Frechet Derivative PAT 1D} and Figure~\ref{FIG:Frechet Derivative PAT 2D} some gradients of the $L^2$ and $W_2$ objective functions, with respect to the absorption coefficient $\sigma$ (i.e. $m=\sigma$), for the inverse diffusion problem we study in Section~\ref{SUBSEC:PAT}, in one- and two-dimensional domains $\Omega=(0, 1)$ and $\Omega=(0, 1)\times(0, 1)$ respectively. The synthetic data, generated by applying the forward operator to the true absorption coefficient and then adding multiplicative random noise, contains roughly $5\%$ of random noise. We intentionally choose initial guesses to be relatively far away from the true coefficient so that the model prediction $f(m)$ is far from the data $g$ to be matched. We are not interested in a direct quantitative comparison between the gradient of the Wasserstein objective function and that of the $L^2$ objective function since we do not have a good basis for the comparison. However, it is evident from these numerical results that the gradient of the Wasserstein functional is smoother, or contains fewer frequencies to be more precise, compared to the corresponding $L^2$ case.

\section{Wasserstein inverse matching for transportation and convolution}
\label{SEC:W2 De-Conv}


Its robustness against high-frequency noise in data, resulted from its smoothing effect we demonstrated in the previous two sections, is not the only reason why $W_2$ is thought as better than $L^2$ for many inverse data matching problems. We show in this section another advantage of the $W_2$ distance in studying inverse matching problems: its convexity with respect to translation and dilation of signals.

\subsection{$W_2$ convexity with respect to affine transformations}
\label{SUBSEC:W2 Transp}

For a given probability density function $\phi$ with finite moments $\bM_1:=\int_{\bbR^d} \bx \phi(\bx) d\bx$ and $\mathbb M_2:=\int_{\bbR^d}|\bx|^2 \phi(\bx) d\bx$, we define:
\begin{equation}\label{EQ:Affine Transform}
	f(m(\bx)):= \dfrac{1}{\sqrt{|\Sigma|}}\phi\big(\Sigma^{-\frac12}(\bx-\lambda \bar\bx)\big).
\end{equation}
where $m:=(\Sigma, \lambda, \bar\bx)$ with $\Sigma\in\bbR^{d\times d}$ symmetric and positive-definite, $\lambda\in\bbR$ and $\bar\bx\in\bbR^d$. This $f$ is simply a translation (by $\lambda\bar\bx$) and dilation (by $\Sigma^{1/2}$) of the function $\phi$. We verify that $\dint_{\bbR^d} f(m(\bx)) d\bx=\dint_{\bbR^d}\phi(\bx) d\bx$. 

Let $g=f(m_g)$ be generated from $m_g:=(\Sigma_g, \lambda_g, \bar\bx_g)$ with $\Sigma_g\in\bbR^{d\times d}$ symmetric and positive-definite. Then we check that the optimal transport map from $f$ to $g$ is given by $\bT(\bx)=\Sigma_g^{1/2}\Sigma^{-1/2}(\bx-\lambda\bar\bx)+\lambda_g\bar\bx_g$. In other words, the function $u(\bx)=\dfrac{1}{2}(\bx-\bar\bx)^t \Sigma_g^{-1/2}\Sigma^{1/2} (\bx-\bar\bx) +\bar\bx_g\cdot \bx$ is a convex solution to the Monge-Amp\`ere equation~\eqref{EQ:Monge-Ampere} with this $(f, g)$ pair. This observation allows us to find that,
\begin{equation}\label{EQ:W2 Affine Transform}
	W_2^2(f, g)=|\lambda\bar\bx-\lambda_g\bar\bx_g|^2+2(\lambda\bar\bx-\lambda_g\bar\bx_g)^t(\Sigma^{1/2}-\Sigma_g^{1/2})\bM_1 + \int_{\bbR^d}|(\Sigma^{1/2}-\Sigma_g^{1/2})\bx|^2 \phi(\bx) d\bx.
\end{equation}
This calculation shows that $W_2^2(f, g)$ is convex with respect to $\lambda\bar\bx-\lambda_g\bar\bx_g$ and $\Sigma^{1/2}-\Sigma_g^{1/2}$ for rather general probability density function $\phi$. 

\begin{figure}[htb!]
	\centering
	\includegraphics[width=0.3\linewidth]{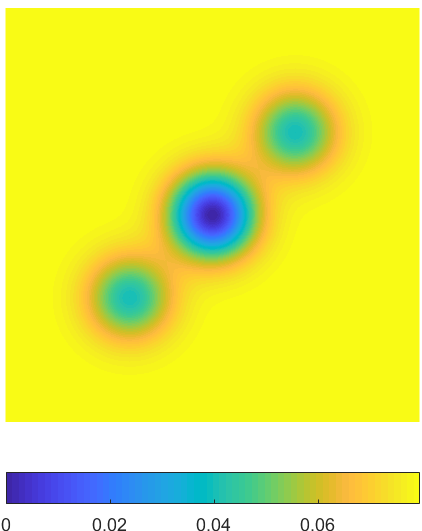}
	\includegraphics[width=0.3\linewidth]{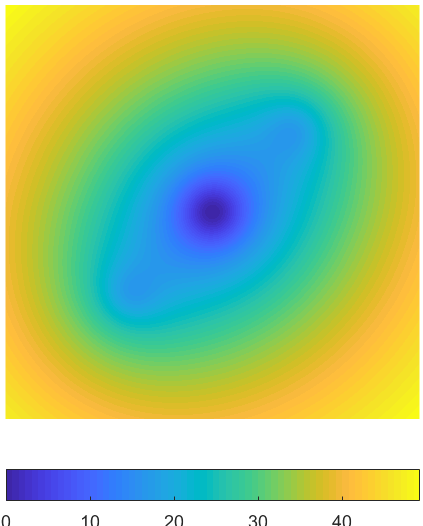}
	\includegraphics[width=0.3\linewidth]{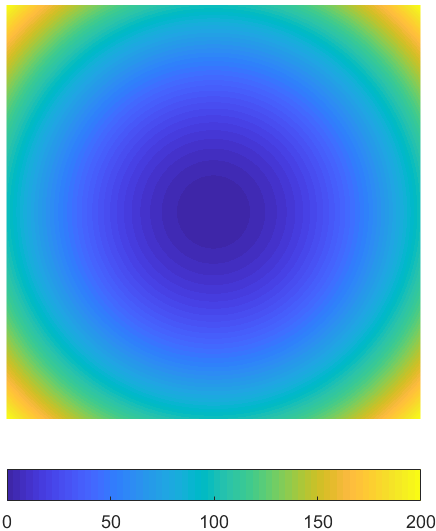}\\
	\includegraphics[width=0.3\linewidth]{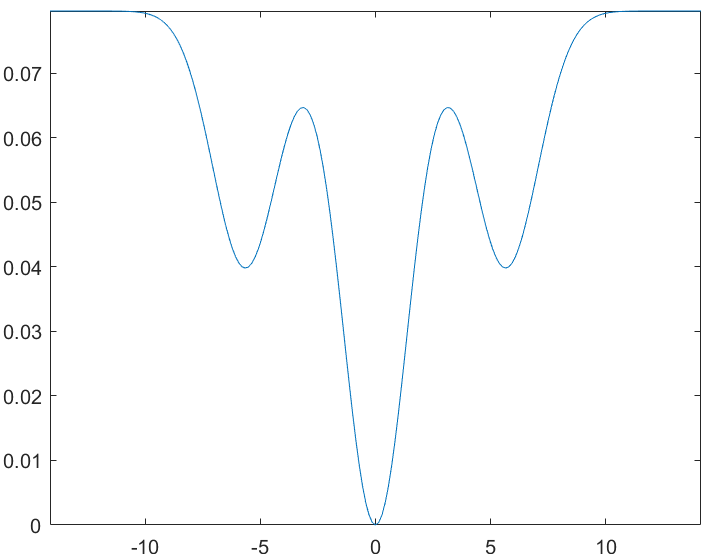}
	\includegraphics[width=0.3\linewidth]{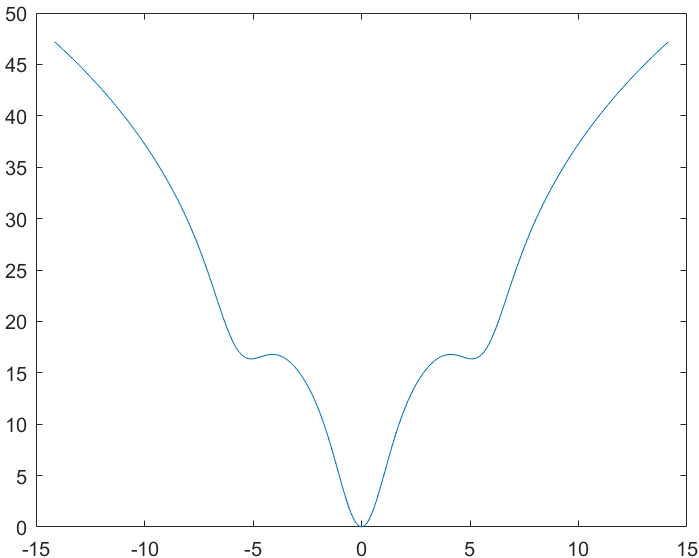}
	\includegraphics[width=0.3\linewidth]{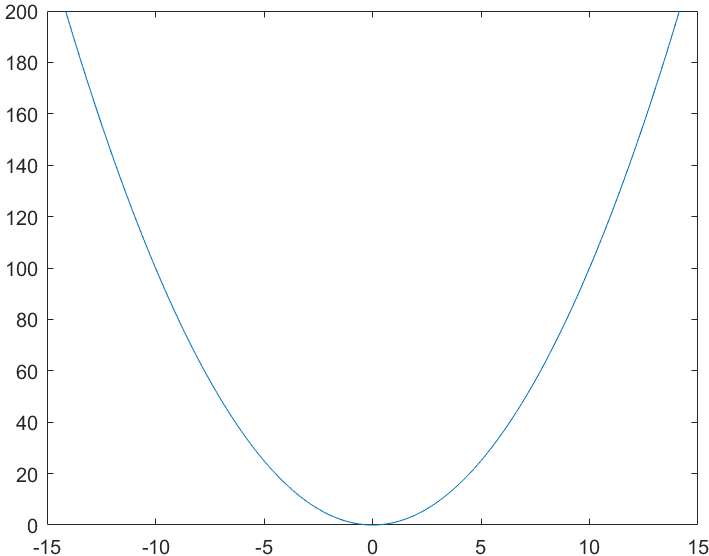}
	\caption{Top row: plots of the $\|f-g\|_{L^2}^2$ (left), $\|f-g\|_{\dot\cH^{-1}}^2$ (middle) and $W_2^2(f, g)$ (right) for the $\phi(\bx)$ given in~\eqref{EQ:phi} in the two-dimensional region $\lambda \bar\bx-\lambda_g\bar\bx_g\in [-10, 10]^2$ with $\Sigma=\Sigma_g=\bI_2$. Bottom row: the corresponding cross-sections along the left bottom to top right diagonal.}
	\label{FIG:Convexity}
\end{figure}
For the purpose of comparison, we recall that the $L^2$ distance between $f$ and $g$ in this case is given by
\begin{multline}\label{EQ:L2 Transl-Dil}
	\|f-g\|_{L^2}^2=(\frac{1}{|\Sigma^{1/2}|}+\frac{1}{|\Sigma_g^{1/2}|})\int_{\bbR^d}\phi^2(\by) d\by \\ 
	-\dfrac{2}{|\Sigma_g^{1/2}|}\int_{\bbR^d} \phi(\by)\phi\big(\Sigma_g^{-1/2}\Sigma^{1/2}\by+\Sigma_g^{-1/2}(\lambda\bar\bx-\lambda_g\bar\bx_g)\big) d\by,
\end{multline}
and the $\dot \cH^{-1}$ distance between $f$ and $g$ in this case is given by
\begin{equation}\label{EQ:dHm1 Transl-Dil}
	\|f-g\|_{\dot\cH^{-1}}^2=\int_{\bbR^d}|\bxi|^{-2} \Big(|\wh \phi(\Sigma^{1/2}\bxi)|^2 +|\wh \phi(\Sigma_g^{1/2}\bxi)|^2 -2\Re\big[\wh\phi(\Sigma^{1/2}\bxi)\overline{\wh\phi}(\Sigma^{1/2}_g\bxi)e^{i(\lambda\bar\bx-\lambda_g\bar\bx_g)\cdot\bxi}\big]  \Big) d\bxi.
\end{equation}
By simple Taylor expansions, we see that $\|f-g\|_{L^2}^2$ and $\|f-g\|_{\dot\cH^{-1}}^2$ are indeed convex when $\Sigma^{1/2}-\Sigma_g^{1/2}$ and $\lambda\bar\bx -\lambda_g\bar \bx_g$ are sufficiently small. However, neither $\|f-g\|_{L^2}^2$ nor $\|f-g\|_{\dot\cH^{-1}}^2$ is globally convex with respect to $\Sigma^{1/2}-\Sigma_g^{1/2}$ or $\lambda\bar\bx -\lambda_g\bar \bx_g$ for a general probability density $\phi(\bx)$. This is one of the major differences between $W_2$, $L^2$ and $\dot\cH^{-1}$. To visualize this advantage of $W_2$, we plot in Figure~\ref{FIG:Convexity} the functionals $W_2^2(f, g)$, $\|f-g\|_{L^2}^2$ and $\|f-g\|_{\dot\cH^{-1}}^2$ in the two-dimensional region $\Omega=[-10, 10]^2$ with fixed $\Sigma=\Sigma_g=\bI_2$ for 
\begin{equation}\label{EQ:phi}
\phi(\bx)=\frac{1}{4\pi}\big[e^{-\frac{1}{2}|\bx-\by_1|^2}+e^{-\frac{1}{2} |\bx-\by_2|^2}\big], \qquad \by_1=(-2, -2),\ \ \by_2=(2, 2).
\end{equation}
The plots clearly demonstrate the advantage of $W_2$ over $L_2$ and $\dot\cH^{-1}$ in terms of its convexity property: while $W_2^2$ is globally convex, $L_2$ and $\dot\cH^{-1}$  are only convex locally. 

In fact, the case of $\phi(\bx)=(2\pi)^{-d/2}e^{-\frac{1}{2} |\bx|^2}$ is well-known in the statistics literature~\cite{DeDe-arXiv19}. In this case, $f$ and $g$ are Gaussian densities with mean-covariance $(\lambda\bar \bx, \Sigma)$ and $(\lambda_g\bar \bx_g, \Sigma_g)$ respectively. The $W_2^2$ distance between them is simplified to
\begin{equation}\label{EQ:W2 Gaussians}
	W_2^2(f, g)=|\lambda\bar\bx-\lambda_g\bar\bx_g|^2+{\rm Tr}(\Sigma+\Sigma_g-2(\Sigma^{1/2}\Sigma_g\Sigma^{1/2})^{1/2}).
\end{equation}
This shows that $W_2^2(f, g)$ is convex with respect to the difference of the mean and variance of the two Gaussian densities. This fact makes the quadratic Wasserstein metric extremely useful for inverse matching of Gaussian densities. 

\subsection{Inverse transport with $W_2$}

The simple calculations we just had can turn out to be very useful in solving some inverse matching problems.

{\bf Transport in homogeneous flow.} Let us consider the transport of a physical quantity $\phi$ in a given uniform flow $\bv$.  The evolution of the quantity is modeled by the following transport equation:
\begin{equation}\label{EQ:Transport}
	\dfrac{\partial\psi}{\partial t} +\bv\cdot \nabla \psi =0, \ \ \mbox{in}\ \ \bbR_+\times\bbR^d, \qquad \psi(0, \bx)=\phi(\bx),\ \ \mbox{in}\ \bbR^d.
\end{equation}
It is straightforward to check that the solution to this transport equation at time $t=\lambda$ is given as
\begin{equation}\label{EQ:Transl-Dil}
	\psi(\lambda, \bx) =\phi(\bx-\lambda \bv).
\end{equation}
For a given function $\phi$, we are interested in finding from the datum $g:= \phi(\bx-\lambda_g \bv_g)$ the unknown flow $\bv$ and the travel distance $\lambda$ by matching the predicted datum $f$ with $g$ under the $W_2$ metric. The result in~\eqref{EQ:W2 Affine Transform} then shows that $W_2^2(f, g)$ is convex with respect to $\lambda\bv$. More precisely,
\begin{equation}\label{EQ:W2 Transl-Dil}
	W_2^2(f, g)=|\lambda\bv-\lambda_g\bv_g|^2.
\end{equation}
Therefore the inverse matching problem of determining $\lambda\bv$ from given data is a convex problem under the $W_2$ metric.  

Note that since the dependence of $W_2(f, g)$ on $\lambda$ and $\bv$ is only through the product, we can generalize our nonlinear model~\eqref{EQ:Transl-Dil} by making the flow more complicated. One example of such generalization is to make the change
\[
	\lambda \bv \to \sum_{j=1}^J \lambda_j\bv_j, 
\]
for any $J$ given flow $\{\bv_j\}_{j=1}^J$. In this case, $W_2^2(f, g)$ is convex with respect to $\lambda_j\bv_j-\lambda_{g_j}\bv_{g_j}$ ($1\le j\le J$).

\medskip

{\bf Reconstruction from projections.} The convexity of $W_2^2$ with respect to translation of signals is also useful in recovering locations of objects in high-dimensional space from (possibly random and noisy) projections. Let $P_j\phi$ be the $j$-th projection of $\phi$ on a collection of $K_j$ coordinates $(x_{k_1},\cdots,x_{k_{K_j}})\in\bbR^{K_j}$, that is,
\[
	(P_j\phi)(x_{k_1},\cdots,x_{k_{K_j}}):=\int_{\bbR^{d-K_j}} \phi(\bx) dx_{k_{K_j}+1}\cdots dx_{k_d}.
\]
We assume that data from $J$ such projections are collected, and that each coordinate in $\bbR^d$ has been projected onto at least once in the $J$ projections (since otherwise the reconstruction will be nonunique). Let $f_j$ and $g_j$ be the $j$-th projections of $\phi(\bx-\bar\bx)$ and $\phi(\bx-\bar\bx_g)$ respectively. Then from~\eqref{EQ:W2 Affine Transform}, we see that the functional
\[
	\Phi(\bar\bx, \bar\bx_g)=\sum_{j=1}^J W_2^2(f_j, g_j) = \sum_{j=1}^J \Pi(j) (\bar x_j-\bar x_{gj})^2
\]
with $\Pi(j)$ the number of times that coordinate $x_j$ is included in the $J$ projections ($1\le \Pi(j)\le J$), is globally convex with respect to each $\bar x_j-\bar x_{gj}$. Therefore, locating an objects from measured projections is a convex problem in the $W_2$ framework.

\subsection{$W_2$-based deconvolution of localized sources}

Let us consider  the linear matching problem~\eqref{EQ:Lin IP} where $A$ is a linear convolution operator defined as:
\begin{equation}\label{EQ:Conv}
	f(m(\bx)) \equiv Am(\bx):=\int_{\bbR^d} K(\bx-\by) m(\by) d\by.
\end{equation}
This type of operators appear in many areas of applications, such as signal and image processing and optical imaging, where $K$ serves as the model of the point spread function of some given physical systems~\cite{BlHa-Book81,Dimri-Book92,GoWo-Book02,NaWo-JOSA10,TuLaSz-EMST01}. The measured signal $f$ can be viewed as output of the system for the input source $m$.

In many applications, $A$ is highly smoothing, meaning that its singular values decay very fast. Inverting $A$ to reconstruct all the information in $m$ is often impossible due to the noise presented in the data. Here we are interested in reconstructing highly localized sources, that is, sources that have their total mass supported in a small part of their supports. Point sources are such sources and are of great importance in many practical applications. In general, let $0<\eps\ll 1$ be given. We introduce
\[
	\cM_{\by,\eps}=\{ m(\bx)\ge 0\ |\  \exists B_\eps(\by) \subseteq \supp(m)\ {\rm s. t.}\ \int_{B_\eps(\by)} m(\bx) d\bx \ge (1-\eps^{d+1})\int_{\bbR^d} m(\bx) d\bx \}
\]
where $B_\eps(\by)$ denote the ball of radius $\eps$ centered at $\by$. Functions in $\cM_{\by, \eps}$ have their total mass concentrated in a ball of radius $\eps$, and are therefore highly localized.

It is straightforward to show the following result.
\begin{theorem}
	Let $f$ and $g$ be generated from~\eqref{EQ:Conv} with $m_f$ and $m_g$ respectively. Assume that $m_f$ and $m_g$ have the same total mass. (i) For any kernel function $K(\bx)$ with finite total mass, if 
		\[
		 	m_{\zeta}(\bx)=\delta(\bx-\bar\bx_\zeta),\ \ \zeta\in\{f, g\}
		\]
	then we have that
	\begin{equation}\label{EQ:W2 Conv Point Source}
	W_2^2(f, g)=|\bar\bx_f-\bar\bx_g|^2.
	\end{equation}
	(ii) For any kernel function $K(\bx)\in\cC^2(\bbR^d)$, if $m_\zeta\in \cC^2(\bbR^d)\cap \cM_{\bar\bx_\zeta,\, \eps}$, then
	\begin{equation}\label{EQ:W2 Conv}
		W_2^2(f, g)=|\bar\bx_f-\bar\bx_g|^2+\cO(\eps^{d+1}).
	\end{equation}
\end{theorem}
\begin{proof}
(i) follows from the fact that in this setting, $\zeta(\bx):=Am_\zeta(\bx)=K(\bx-\bar\bx_{\zeta})$, $\zeta\in\{f, g\}$, which is only a translation of $K$ by $\bar\bx_\zeta$. To prove (ii), we first observe that inside the $B_\eps(\bar\bx_\zeta)$ we can write
\[
	m(\by)=m(\bar\bx_\zeta)+(\by-\bar\bx_\zeta)\cdot \nabla m(\bar\bx_\zeta)+\cO(|\by-\bar\bx_\zeta|^2),
\]
\[
	K(\bx-\by)=K(\bx-\bar\bx_\zeta) + (\bar\bx_\zeta-\by)\cdot \nabla K(\bx-\bar\bx_\zeta) +\cO(|\bar\bx_\zeta-\by|^2).
\]
Therefore we have
\begin{multline*}
	\zeta(\bx)=\int_{\bbR^d} K(\bx-\by) m(\by) d\by = \int_{B_\eps(\bar\bx_\zeta)} K(\bx-\by) m(\by) d\by +\cO(\eps^{d+1}) \\ 
	= m(\bar\bx_\zeta)K(\bx-\bar\bx_\zeta)|B_\eps(\bar\bx_\zeta)| + \cO(\eps^{d+1}).
\end{multline*}
This gives that $f$ and $g$ are simply perturbations of the translations of $K$, which then leads to the desired result.
\end{proof}

{\bf Deconvolution of highly localized sources.} The first important consequence of this result, described in part (i) of the theorem, is that the deconvolution from datum $g$, with an \emph{arbitrary} kernel $K$, to recover the location of a point source is a convex problem under the $W_2$ metric; see~\eqref{EQ:W2 Conv Point Source}. This is a simple but not obvious observation because once the source is parameterized in terms of the location, the convolution problem becomes nonlinear (with respect to the location). Part (ii) of the theorem generalizes point source to any highly concentrated sources by saying that to the leading order, the $W_2$ metric allows us to have a convex objective function in recovering the location of a highly localized object. 

{\bf Deconvolution from diffusive environments.} One aspect of the result that is surprising is that it does not require the convolution kernel to take a specific form, for instance, to come from a physical system governed by the transport phenomenon such as what we described in the previous subsection. This allows us to study deconvolution with smoothing kernels that could serve as models for the propagation of the information in $m$ in a diffusive environment. A kernel of particular importance is the Gaussian kernel
\begin{equation}\label{EQ:Kernel Gauss}
 	K(\bx)=\dfrac{\Lambda_d}{|\Sigma_K|^{1/2}} e^{-\frac{1}{2}\bx^t\Sigma_K^{-1} \bx},
\end{equation}
with $\Sigma_K\in\bbR^{d\times d}$ symmetric and positive-definite and the constant $\Lambda_d=(2\pi)^{-d/2}$. For this specific kernel, we could make the asymptotic calculation in the theorem more precise. Let us take the function $m$ to be of Gaussian type, as an approximation to the localized source we discussed in the theorem (when we make the covariance matrix small). More precisely, we take
\[
	m_{\zeta}(\bx)=\dfrac{\Lambda_d}{|\Sigma_{m_\zeta}|^{1/2}} e^{-\frac{1}{2}(\bx-\bx_{m_\zeta})^t \Sigma_{m_\zeta}^{-1}(\bx-\bx_{m_\zeta})},\ \ \zeta\in\{f, g\}
\]
with $\Sigma_{m_f}$ and $\Sigma_{m_g}$ symmetric and positive-definite. Then it is straightforward to verify that $f$ and $g$ are of the form:
\[
	\zeta(\bx)=\dfrac{\Lambda_d}{|\Sigma_K+\Sigma_{m_\zeta}|^{1/2}} e^{-\frac{1}{2}(\bx-\bx_{m_\zeta})^t(\Sigma_K+\Sigma_{m_\zeta})^{-1}(\bx-\bx_{m_\zeta})}.
\]
Following~\eqref{EQ:W2 Gaussians}, we have that
\begin{equation*}
	W_2^2(f, g)=|\bx_f-\bx_g|^2+{\rm Tr}(\Sigma_f+\Sigma_g-2(\Sigma_f^{1/2}\Sigma_g\Sigma_f^{1/2})^{1/2}),
\end{equation*}
where $\Sigma_\zeta=\Sigma_K+\Sigma_{m_\zeta}$. To mimic the case of localized source, we perform the rescaling $\Sigma_{m_f} \to \eps_f^2\Sigma_{m_f}$ and $\Sigma_{m_g} \to \eps_g^2\Sigma_{m_g}$ with $\eps_f$ and $\eps_g$ both small (but not necessarily the same). Then $W_2^2(f, g)$ simplifies to
\begin{multline}\label{EQ:W2 Conv Expansion}
	W_2^2(f, g) = |\bx_f-\bx_g|^2+\dfrac{1}{2} {\rm Tr} \Big(\Sigma_K(\eps_f^2\Sigma_K^{-1} \Sigma_{m_f}-\eps_g^2 \Sigma_K^{-1}\Sigma_{m_g})^t(\eps_f^2\Sigma_K^{-1} \Sigma_{m_f}-\eps_g^2 \Sigma_K^{-1}\Sigma_{m_g})\Big) \\
	+\dfrac{3}{8} \eps_g^4{\rm Tr} \Big(\Sigma_K (\Sigma_K^{-1}\Sigma_{m_g})^t(\Sigma_K^{-1}\Sigma_{m_g})\Big) + \cO(\eps_f^6+\eps_g^6+\eps_f^2\eps_g^4+\eps_f^4\eps_g^2).
\end{multline}
Taking the shapes of $m_f$ and $m_g$ to be the same, that is, $\Sigma_{m_f}=\Sigma_{m_g}$, we see that $W_2^2(f, g)$ is convex with respect to both $\bx_{m_f}-\bx_{m_g}$ and $\eps_f^2-\eps_g^2$ at the leading orders. This fact gives us the opportunity to deconvolve from Gaussian kernels to reconstruct stably the location and relative size of localized sources. Note, however, that due to the presence of the trace operator, $W_2^2(f, g)$ is less sensitive to the shape (or anisotropy) of the source encoded in its covariance matrix, compared to the size of the source.
	
\section{Numerical simulations}
\label{SEC:Num}

In this section, we present some numerical simulations to demonstrate some of the main phenomena that we analyzed in this work. We consider two different inverse data matching problems.

\subsection{Deconvolution under the $W_2$ metric}
\label{SUBSEC:Num Poisson}

We first show some simulation results on the deconvolution problem in the one-dimensional setting. Our numerical simulations can only be done in a finite domain, so we set $\Omega=[-\ell,\ \ell]$. The forward operator is the convolution on $\Omega$:
\begin{equation}\label{EQ:Conv Finite}
	f(m(x))=Am(x):=\int_{-\ell}^\ell K(x-y) m(y) dy,
\end{equation}
for a kernel $K(x)> 0$ with finite integral $\int_{-2\ell}^{2\ell} K(x) d x=\aver{K}$. We assume that the data $g$ is generated from a true function $m^*(x)\ge 0$ such that $\int_{-\ell}^\ell m^*(x) d x =\aver{m^*}$. This will ensure that $\int_{-\ell}^\ell g(x) dx=\aver{K}\aver{m^*}$. We will consider inverse matching with both noiseless and noisy data. In the case of noisy data, we add multiplicative noise to $g$ in a way such that $\int_{-\ell}^\ell g_\delta(x) dx=\aver{m^*}$.
\begin{figure}
	\centering
	\includegraphics[width=0.32\linewidth]{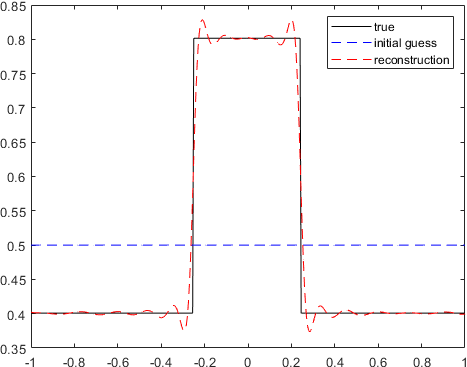}
	\includegraphics[width=0.32\linewidth]{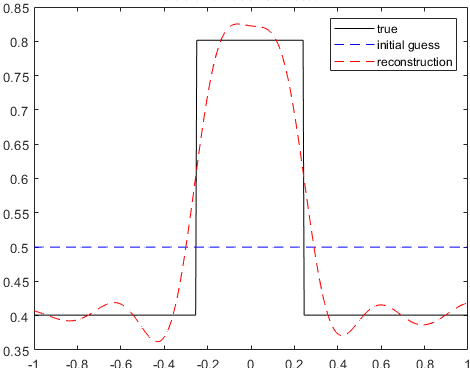}
	\includegraphics[width=0.32\linewidth]{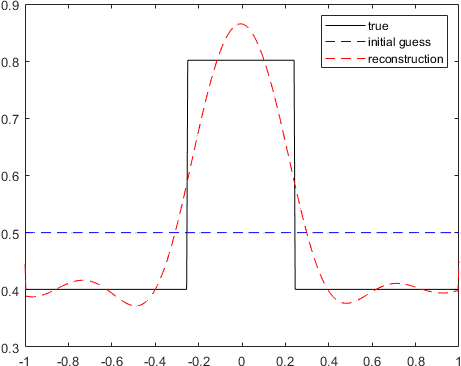}
	\includegraphics[width=0.32\linewidth]{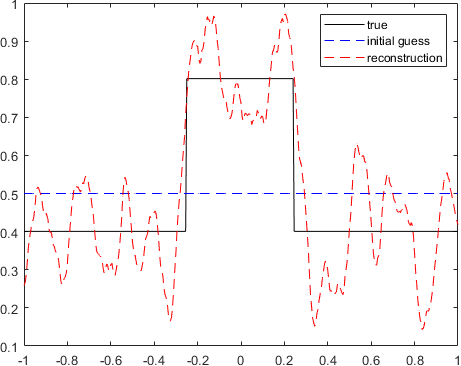}
	\includegraphics[width=0.32\linewidth]{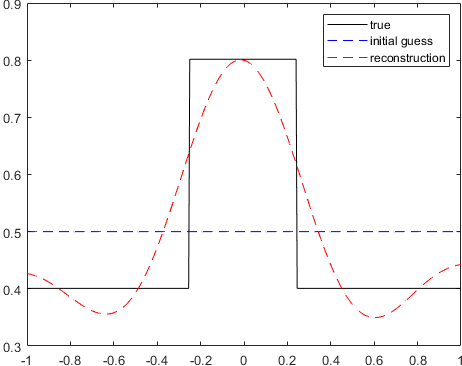}
	\includegraphics[width=0.32\linewidth]{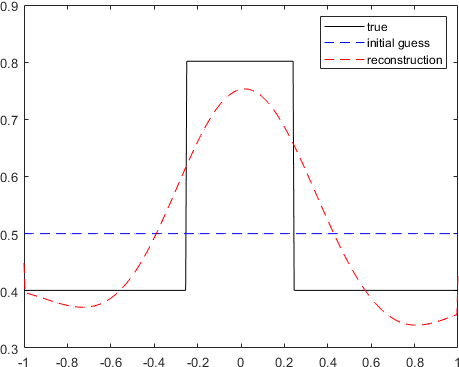}
	\caption{Deconvolution with a Laplacian kernel with the $L^2$ (left), $\dot\cH^{-1}$ (middle), and $W_2$ (right) metrics. Top row: deconvolution with noise-free data; Bottom row: deconvolution with data containing respectively $2\%$, $10\%$, and $10\%$ random noise.}
	\label{FIG:DeConv Gauss}
\end{figure}

In Figure.~\ref{FIG:DeConv Gauss} we show deconvolution results for the Laplacian kernel $K_L(x)=e^{-\ell |x|}$ with the $L^2$, $\dot\cH^{-1}$ and $W_2$ metrics, that is, using the objective functions $\Phi_{\cH^0}(m)$, $\Phi_{\dot\cH^{-1}}(m)$ and $\Phi_{W_2}(m)$. This set of results show clearly the smoothing effect of the quadratic Wasserstein metric: in both the noise-free and noisy data cases, $W_2$ give smoothed out reconstruction compared to its $L^2$ counterpart. The order of the smoothing is very similar to that of the $\dot\cH^{-1}$ norm, although the exact reconstructions are quite different. Compared to $\dot\cH^{-1}$, $W_2$ reconstructions seem to have slightly more variations. Both the $\dot\cH^{-1}$ and the $W_2$ reconstructions can tolerate strong high-frequency random noise, while the $L^2$ reconstructions break down quickly at the relatively low noise level. One can certainly add a regularization mechanism to stabilize the $L^2$ reconstructions here. However, our objective here is not to say that $W_2$ reconstruction is better than $L^2$ reconstruction, but to demonstrate the smoothing effect of the quadratic Wasserstein metric. Therefore we do not include any explicit regularization to the $L^2$ reconstruction (although our discretization of the forward problem, the adjoint problem as well as the mismatch between the forward and adjoint discretization, indeed introduce a regularization effect).

\begin{figure}
	\centering
	\includegraphics[width=0.32\linewidth]{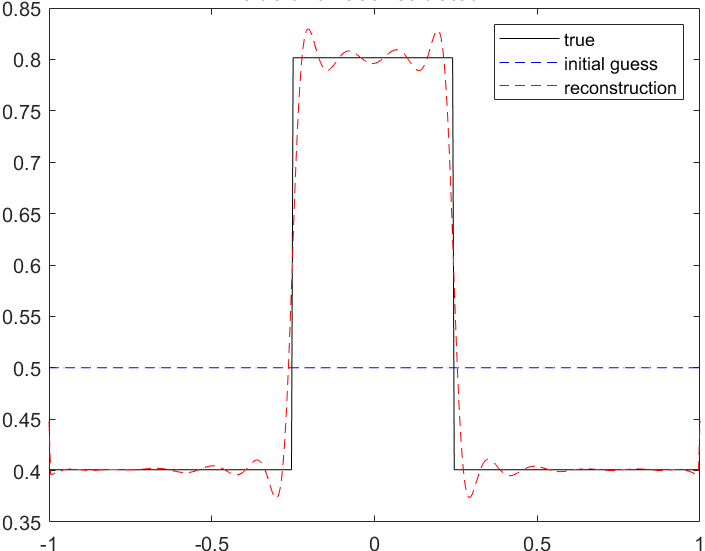}
	\includegraphics[width=0.32\linewidth]{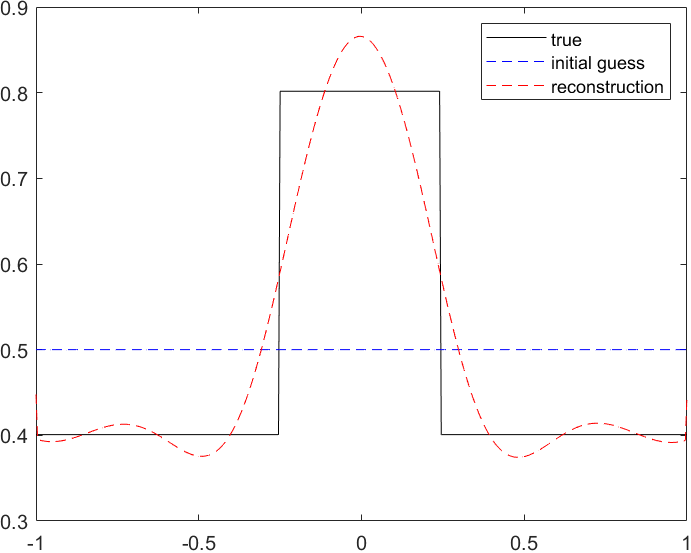}
	\includegraphics[width=0.32\linewidth]{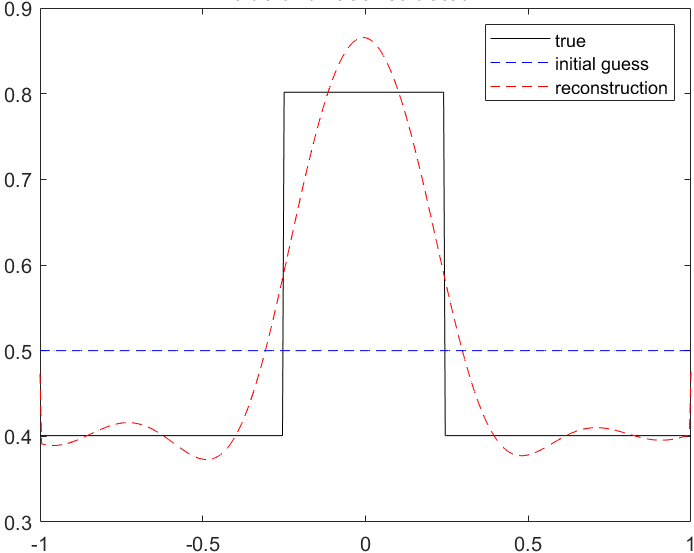}\\
	\includegraphics[width=0.32\linewidth]{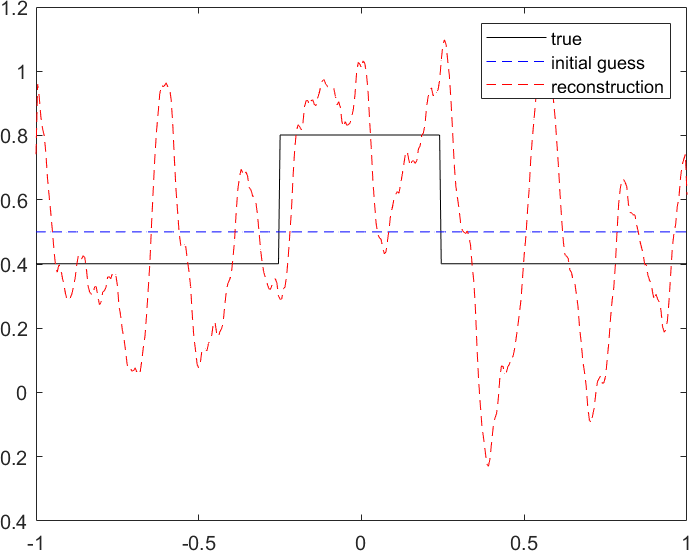}
	\includegraphics[width=0.32\linewidth]{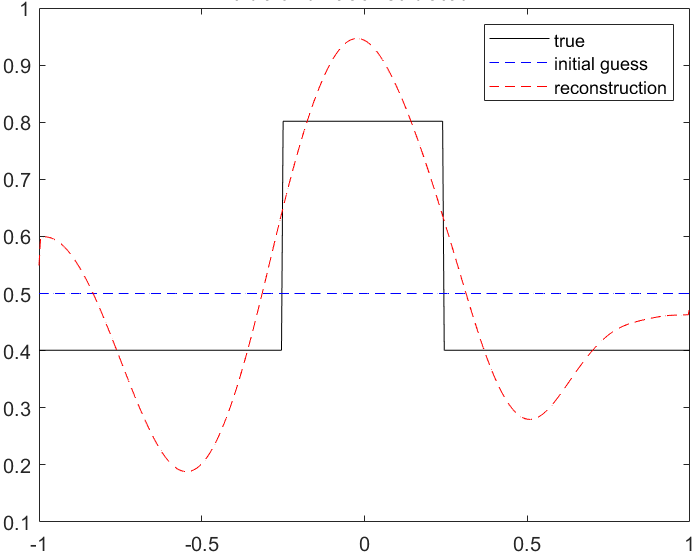}
	\includegraphics[width=0.32\linewidth]{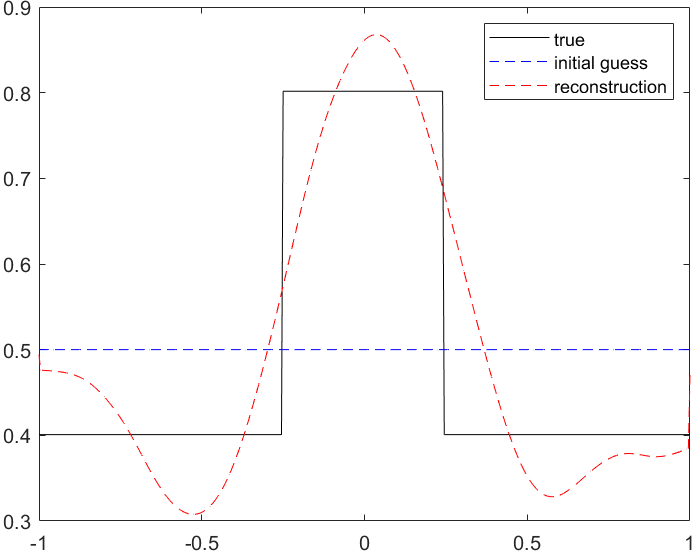}
	\caption{Deconvolution with the kernel $K_I(x)=\frac{1}{1+|x|}$ with the $L^2$ (left), $\cH^{-1}$ (middle), and $W_2$ (right) metrics. Top row: deconvolution with noise-free data; Bottom row: deconvolution with data containing respectively $2\%$, $10\%$, and $10\%$ random noise.}
	\label{FIG:DeConv Exp}
\end{figure}
We have done extensive numerical experiments on this problem with data at different noise levels and different convolution kernels. The same type of smoothing effect are observed in all the numerical tests with the quadratic Wasserstein metic. Interested readers are refer to~\cite{DiDuRe-Prep19} for more detailed description of such results. As an example, we show in Figure~\ref{FIG:DeConv Exp} deconvolution results from the kernel $K_I(x)=\frac{1}{1+|x|}$. 

\subsection{Inverse matching for an elliptic PDE problem}
\label{SUBSEC:PAT}

The second set of numerical simulations focuses on an inverse coefficient problem for the diffusion equation. Let $\Omega\subset\bbR^d$ be a bounded domain with boundary $\partial\Omega$ and consider the following second-order elliptic partial differential equation with Robin boundary condition:
\begin{equation}\label{EQ:Diff}
	-\nabla\cdot \gamma u(\bx) + \sigma(\bx) u(\bx) = 0,\ \ \mbox{in}\ \ \Omega, \qquad \bn\cdot \gamma \nabla u(\bx)+\kappa u=h(\bx),\ \ \mbox{on}\ \ \partial\Omega.
\end{equation}
Physically this equation can be used to describe the diffuse of particles, generated from source $h(\bx)\ge 0$, in an absorbing environment, such as propagation of near infra-red photons in biological tissues, with $\gamma(\bx)>0$ and $\sigma(\bx)> 0$ respectively the diffusion and absorption coefficients of the environment. 

For applications in quantitative photoacoustic imaging~\cite{BaRe-IP11,ReGaZh-SIAM13}, we are interested in reconstructing the absorption coefficient $\sigma$ in the equation~\eqref{EQ:Diff} from data that we measure in an experiment of the following form:
\begin{equation}\label{EQ:Data}
	f(\sigma):= \sigma(\bx) u(\bx), \qquad \bx\in\Omega
\end{equation}
where $u(\bx)$ is the solution to the diffusion equation. Due to the fact that the diffusion solution $u$ depends on the unknown $\sigma$ in a nonlinear way, this inverse problem is nonlinear. 

\begin{figure}
	\centering
	\includegraphics[width=0.32\linewidth]{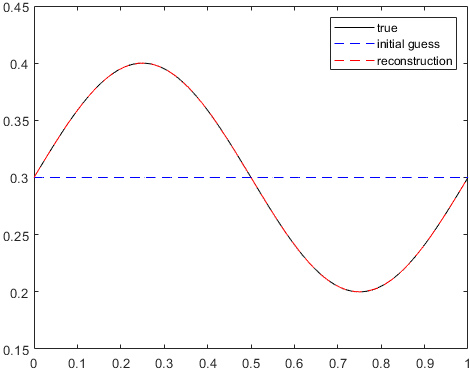}
	\includegraphics[width=0.32\linewidth]{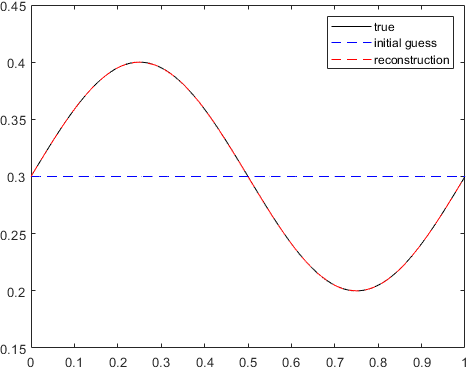}
	\includegraphics[width=0.32\linewidth]{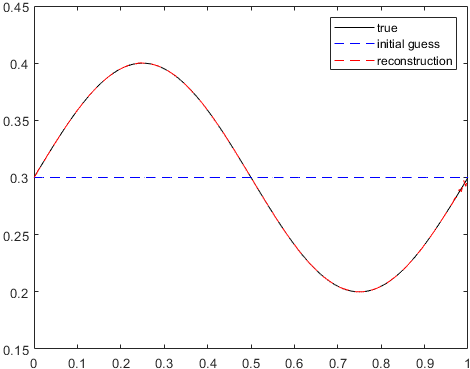}
	\includegraphics[width=0.32\linewidth]{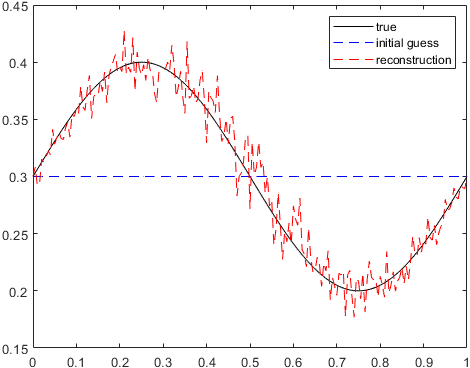}
	\includegraphics[width=0.32\linewidth]{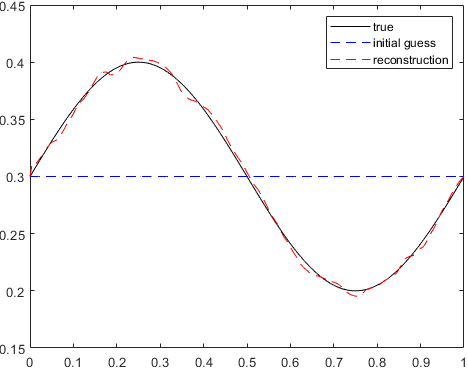}
	\includegraphics[width=0.32\linewidth]{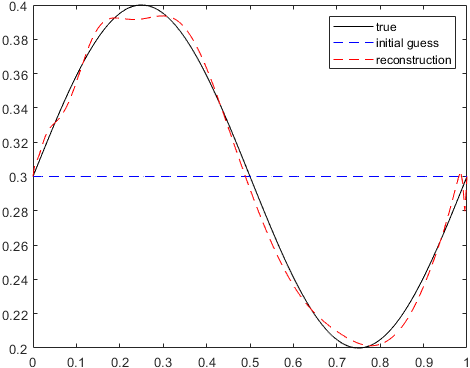}
	\caption{Inversion for $\sigma$ in $\Omega=(0, 1)$ with noise-free (top row) and noisy data (bottom row) under the $L^2$ (left), $\dot\cH^{-1}$ (middle) and $W_2$ (right) metrics. The noise level in the bottom row is $12\%$ for each case.}
	\label{FIG:PAT 1D}
\end{figure}
Let us assume that $\Omega$ is smooth, so that the diffusion equation~\eqref{EQ:Diff} admits a unique solution that is sufficiently regular when $\gamma$, $\sigma$ and $h$ are sufficiently regular. Moreover, from standard theory of elliptic partial differential equations~\cite{GiTr-Book00}, we conclude that $u$ is non-negative and bounded from above when the boundary source $h(\bx)$ is so. Therefore $f(\sigma)$ is nonnegative and bounded from above. We therefore can treat $f(\sigma)$ as a probability density for a given $\sigma$.

We first show in Figure~\ref{FIG:PAT 1D} some typical inversion results for the problem in the one-dimensional case where $\Omega=(0, 1)$. For simplicity, we set $\gamma=0.02$ in all the simulations we present here. Our extensive numerical tests show that the exact value of $\gamma$ has only a negligible impact on the inversion results. We again observe the smoothing effect of the $W_2$ metric as in the deconvolution example, and this smoothing effect makes the inversion very robust again high-frequency random noise in the data. This smoothing effect is independent of the dimension of the spatial domain, as seen from the analysis in the previous sections of the paper and observed in two-dimensional inversion results shown in Figure~\ref{FIG:PAT 2D}.
\begin{figure}
	\centering
	\includegraphics[width=0.25\linewidth]{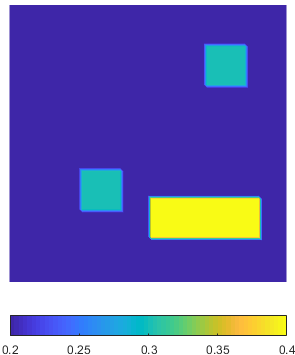}
	\includegraphics[width=0.24\linewidth]{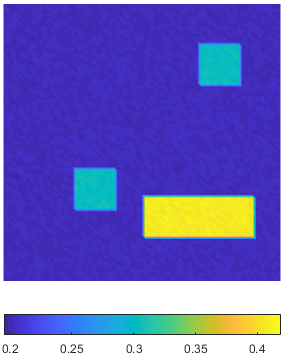}
	\includegraphics[width=0.24\linewidth]{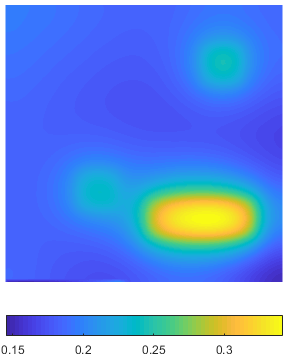}
	\includegraphics[width=0.24\linewidth]{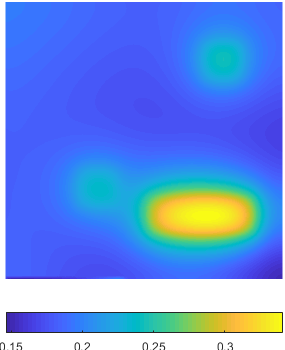}
	\caption{Inversion for $\sigma$ in two-dimensional case in the domain $\Omega=(0, 1)\times(0, 1)$ with data containing $10\%$ random noise. Shown from left to right are the true coefficient, the reconstructions with the $L^2$, $\dot\cH^{-1}$ and $W_2$ metrics respectively.}
	\label{FIG:PAT 2D}
\end{figure}

\section{Concluding remarks}
\label{SEC:Concl}

In this work, we performed analytical and computational studies on the effect of the quadratic Wasserstein metric in solving linear and nonlinear inverse data matching problems. We demonstrate, through analysis in the asymptotic regime, that the quadratic Wasserstein metric has a smoothing effect on the inversion results, that is, at a given noise level, matching results with $W_2$ is smoother than those with the $L^2$ metric. One can see this smoothing effect at the true inverse solution or at a local minimizer of the objective function (which could be far away from the true matching solution) when a numerical minimization approach is used to solve the matching problem. The order of the smoothing effect is the same as that of the $\dot\cH^{-1}$ metric. 

The smoothing effect of the quadratic Wasserstein metric in the non-asymptotic regime indicates that the landscape of the Wasserstein objective function and that of the $L^2$ objective function are quite different. In other words, if we solve the inverse matching problems with numerical minimization, the trajectory of the minimizing sequence based on the $L^2$ metric and that based on the Wasserstein metric is different. Even if the two minimization algorithms start at the same initial guess and converge to the same minimizer, their paths between the starting point and the final point are different. Finding good ways to precisely characterize the differences between the landscapes of the two minimization problems will be the subject of future works.

Even though the smoothing effect of $W_2$ is very similar to that of the $\dot\cH^{-1}$ metric, we showed, through analyzing some simple finite-dimensional inverse matching problems, such as the deconvolution of point sources from given kernels, that $W_2$ has better convexity than $\dot\cH^{-1}$ and the $L^2$ norm in the non-asymptotic regime. Characterizing such advantages in more practically useful situations would be an interesting project.

The quadratic Wasserstein metric we studied in this paper requires that the two data to be compared, $f$ and $g$, be both probability measures (and therefore be nonnegative) and they have the same total mass. This requirement severely limits the applicability of the Wasserstein metric to more general data matching problems. To use the Wasserstein metrics for data that are not nonnegative, for instance in data matching for seismic imaging applications~\cite{ChChWuYa-JCP18,EnFr-CMS14,EnFrYa-CMS16,MeBrMeOuVi-GJI16,MeBrMeOuVi-IP16,YaEn-Geophysics18,YaEnSuFr-Geophysics18}, one can perform some rescaling on the data to make them nonnegative. It seems, at least on the computational level, that the main features of the Wasserstein metrics are preserved by reasonable rescaling strategies such as $f \mapsto \alpha+\beta f$ and $f \mapsto e^{\beta f}$ ($\alpha$ and $\beta$ being appropriately chosen constants). There have been also attempts to generalize the original optimal transport framework to deal with signed data~\cite{AmMaSe-AIHP11,Mainini-JMS12}, although these have remained mostly on the theoretical level so far. To compare two nonnegative signals that do not have the same total mass, one can either use a normalization procedure, for instance by dividing the signals with their respective total mass so that they both have total mass $1$ (c.g.~\cite{ChChWuYa-JCP18,EnFrYa-CMS16,MeBrMeOuVi-IP16,YaEnSuFr-Geophysics18}), or consider generalized Wasserstein distances induced by unbalanced optimal transport. In terms of the later, the Wasserstein-Fisher-Rao metric~\cite{ChPeScVi-FCM18B,MeAlBrMeOuVi-Geophysics18} has been studied extensively in the literature. Analyzing the behavior of the quadratic Wasserstein metric in the aforementioned more complicated setups is an interesting topic for future research.

\section*{Acknowledgments}

We would like to thank the anonymous referees for providing useful comments on the first draft of this work and pointing out several related references. Their inputs helped us improve the quality of this work. This work is partially supported by the National Science Foundation through grants DMS-1620396, DMS-1620473, DMS-1913129, and DMS-1913309.



\end{document}